\documentclass{amsart}
\usepackage{comment}
\usepackage{tikz-cd}
\usepackage{xcolor}
\newbox\xratbelow
\newbox\xratabove
\usepackage{mathabx}
\usepackage{cancel}
\usepackage{geometry}
\usepackage{float}
\usepackage{mathabx}
\usepackage{hyperref}
\newtheorem{theorem}{Theorem}[section]
\newtheorem{corollary}[theorem]{Corollary}

\newtheorem{lemma}[theorem]{Lemma}
\newtheorem{proposition}[theorem]{Proposition}

\theoremstyle{definition}
\newtheorem{definition}[theorem]{Definition}
\newtheorem{remark}[theorem]{Remark}
\newtheorem*{notation}{Notations}

\newtheorem{hypothesis}{Hypothesis}

\newcommand{\xlongarrow}[2][]{%
\setbox\xratbelow=\hbox{\ensuremath{\scriptstyle #1}}%
\setbox\xratabove=\hbox{\ensuremath{\scriptstyle #2}}%
\pgfmathsetlengthmacro{\xratlen}{max(\wd\xratbelow, \wd\xratabove) + .6em}%
\mathrel{\tikz [baseline=-.75ex]
  \draw (0,0) edge[commutative diagrams/rightarrow] node[below] {\box\xratbelow}
  node[above] {\box\xratabove}
  (\xratlen,0) ;}} 

\newcommand{\Lim}[1]{\raisebox{0.5ex}{\scalebox{0.8}{$\displaystyle \lim_{#1}\;$}}}

\def\dif{\mathop{}\hphantom{\mskip-\thinmuskip}\mathrm{d}}%
\let\daccent\d
\let\d\relax
\newcommand\d{\ifmmode\dif\else\expandafter\daccent\fi}

\begin{document}

\title {Singular dynamics for discrete weak K.A.M. solutions of exact twist maps}

\author{Jianxing Du$^{\dag}$}
\address{School of Mathematical Sciences, Beijing Normal University, No. 19, XinJieKouWai St., HaiDian District, Beijing 100875, P. R. China}
\email{jianxingdu@mail.bnu.edu.cn\\ jxdu000@gmail.com}

\author{Xifeng Su$^{\dag}$}
\address{School of Mathematical Sciences, Laboratory of Mathematics and Complex Systems (Ministry of Education)\\
	Beijing Normal University,
	No. 19, XinJieKouWai St., HaiDian District, Beijing 100875, P. R. China}
\email{xfsu@bnu.edu.cn\\ billy3492@gmail.com}

\date{}



\thanks{The datasets generated during and/or analyzed during the current study are available from the corresponding author on reasonable request.}

\begin{abstract}  
For any exact twist map $f$ and any cohomology class $c\in\mathbb{R}$, let $u_c$ be any associated discrete weak K.A.M. solution, 
and we introduce an inherent Lipschitz dynamics $\Sigma_+$ given by the discrete forward Lax-Oleinik semigroup. 
We investigate several properties of $\Sigma_+$ and show that  the non-differentiable points of $u_c$ are globally propagated and forward invariant by $\Sigma_+$.
In particular, such propagating dynamics possesses the same rotation number $\alpha'(c)$ as the associated Aubry-Mather set at cohomology class $c$.

As applications, we provide via $\Sigma_+$ {a discrete analogue of Bernard's regularization theorem \cite{Ber07} and} a detailed exposition of  Arnaud's observation \cite{Arnaud_2011}. Furthermore, we construct and analyze the corresponding dynamics  on the full pseudo-graphs of discrete weak K.A.M. solutions.

\end{abstract}

\maketitle


\section{Introduction}
In the 1890 volume of Acta Mathematica, H. Poincar\'e introduce some of the basic ideas about twist maps of the annulus while he researched on the stability of orbits of the three body problem.

We focus on some recent developments on the twist maps. The celebrated Aubry-Mather theory developed independently by S. Aubry \cite{Aubry83, aubry1983discrete} and J. Mather \cite{Mather82} provides a novel approach to the study of the dynamics of twist diffeomorphisms of the annulus (a weak analogue for twist homeomorphisms can be found in \cite{KO97}). Especially, the existence of action minimizing sets (referred as Aubry-Mather sets) --which could be invariant circles or invariant Cantor sets etc--is established. 
Moreover, one may have
\begin{itemize}
\item every Aubry-Mather set is a subset of a homotopically nontrivial  circle that is a graph over the zero section,

\item every Aubry-Mather set is associated a rotation number and

\item all the Aubry-Mather sets are vertically ordered in the annulus according to the rotation numbers. See e.g. \cite{Golebook}.
\end{itemize}

We concentrate on those Aubry-Mather sets with gaps and intend to study how these gaps evolve. To this aim, we will introduce a singular dynamics such that 
 the gaps are forward invariant by  such dynamics.

Let $f$ be an exact twist map on the annulus $\mathbb{T}\times \mathbb{R}$ and $F$ be its lift on its universal cover $\mathbb{R}^2$. 
Because of the exactness of $f$,  one can associate  $F$ with  a $C^2$ generating function $S: \mathbb{R}^2\rightarrow\mathbb{R}$. 

Let $c\in \mathbb{R}$  be some cohomology class\footnote{In the present paper, the cohomology class $c$ is fixed at the beginning and one can assume $c=0$ for simplicity. The reason why we insist writing $c$ explicitly in the introduction is that we will apply the results here for future work, such as establishing $c$-equivalence tools (see \cite{CY04}) for searching for diffusion orbits in the discrete setting.} 
of $\mathbb{T}$. One can also define the lift $F_c$ and the corresponding generating function $S_c$ at cohomology class $c$ given by
\[
F_c(x, p)= F(x, p+c) - (0, c) \text{ and } S_c(x,y) = S(x, y) + c(x-y) \text{ respectively}. 
\]
We then define respectively the  \emph{discrete backward and forward Lax-Oleinik operators} $T^-$ and $T^+$ by 
\begin{align*}
  T^- [u] (x) := &\inf_{y \in \mathbb{R}} \big\{ u(y) + S_c (y, x) \big\}, \qquad \forall x\in\mathbb{R}\\
  T^+ [u] (x) := &\sup_{y \in \mathbb{R}} \big\{ u(y) - S_c (x, y) \big\}, \qquad \forall x\in\mathbb{R}
\end{align*}
for every continuous $1$-periodic function $u:\mathbb{R}\rightarrow\mathbb{R}$. Then one can prove that there exists a unique constant $\overline{S}_c\in\mathbb{R}$ such that the equation $u=T^-[u]-\overline{S}_c$ admits at least one solution $u_c$ (see for example \cite{GT11, SuT18}). This min-plus eigenvalue problem in ergodic optimization gives rise to the following two notions in celebrated Mather-Fathi theory (see \cite{Mather91,Fathi97}) for Tonelli Lagrangian dynamics:
\begin{itemize}
\item \emph{Mather's $\alpha$-function}: $H^1(\mathbb{T},\mathbb{R})\rightarrow \mathbb{R}, c\mapsto -\overline{S}_c$, which is differentiable everywhere and so  the corresponding rotation number equals $\alpha'(c)$.
\item a \emph{discrete weak K.A.M. solution} $u_c$ for $f$ at cohomology class $c$, which is semi-concave and so  differentiable almost everywhere. 
\end{itemize}

%
%
%
%
We now state our first result below:

\begin{theorem}\label{thm:main}
  Let $f$ be a twist map of the annulus and $c\in \mathbb{R}$ be some cohomology class of $\mathbb{T}$. Let $u_c: \mathbb{R}\rightarrow\mathbb{R}$ be a discrete weak K.A.M. solution for $f$ at cohomology class $c$. Then the optimal forward map $$\Sigma_+:\mathbb{R}\rightarrow\mathbb{R},\ x\mapsto \mathop{\arg\max}_{y\in \mathbb{R}} \{u_c(y) - S_c(x_n, y) \}$$ is single-valued.
 Moreover,  the map $\Sigma_+$ satisfies the following properties: 
 \begin{enumerate}
    \item [(i)] $\Sigma_+$ is a lift of a circle map with degree $1$, i.e., $\Sigma_+(x+1)=\Sigma_+(x)+1$;
    \item [(ii)] $\Sigma_+$ is  non-decreasing and Lipschitz continuous;
    \item [(iii)] $\Sigma_+$ propagates singularities of $u_c$, i.e., if x is not a differentiable point of $u_c$, neither is $\Sigma_+(x)$; for simplicity, we denote by $\operatorname{Sing}(u_c)$ the set of all the non-differentiable points of $u_c$;
    \item [(iv)] $\Sigma_+$ admits a rotation number for positive iterations, that is, 
    \begin{align*}
      \rho(\Sigma_+):=\lim_{n\to+\infty}\frac{1}{n}\left( \Sigma_{+}^n(x)-x \right)
    \end{align*}
  exists for all $x\in\mathbb{R}$ and is independent of $x$. In fact, $\rho(\Sigma_+) = \alpha'(c)$.
  \end{enumerate}
\end{theorem}

As an application of the above properties of the singular dynamics, we immediately obtain 
 that the singularities of any discrete weak K.A.M. solution can be globally propagated with a rotation number, which answers a question of Wei Cheng \cite{Wei21}.
\begin{theorem}\label{thm:packaged}
 Let $f$ be a twist map of the annulus and $c\in \mathbb{R}$ be some cohomology class of $\mathbb{T}$. Let $u_c: \mathbb{R}\rightarrow\mathbb{R}$ be a discrete weak K.A.M. solution for $f$ at cohomology class $c$.
Let $x_0 \in \operatorname{Sing}(u_c)$ and one can define the one-sided sequence $\{x_n\}_{n\in\mathbb{N}}$ iteratively by:
\[
x_{n+1} :=  \mathop{\arg\max}_{y\in \mathbb{R}} \{u_c(y) - S_c(x_n, y) \}  \qquad \text{ for any } n \in \mathbb{N}.
\]
Then, we have
   \begin{itemize}
   \item [(a)] the whole one-sided sequence $\{x_n\}_{n\in\mathbb{N}} \subset \operatorname{Sing}(u_c)$;
  \item [(b)]  the limit $\rho_c:=\Lim{n\to+\infty}\frac{x_n-x_0}{n}$ exists and is independent of $x_0$;
  \item [(c)]  $\rho_c$ equals to the rotation number $\alpha'(c)$ of the associated Aubry-Mather set.
  \end{itemize}
\end{theorem}

Note that  item (a) of Theorem~\ref{thm:packaged} could be regarded as a discrete version of \cite[Lemma 3.2]{CannarsaC} (see also \cite{Zhang20}). To the best of our knowledge, items (b),(c) are new in the literature of twist maps.

As an immediate application, we now investigate the relation between the number of the singular set $\operatorname{Sing}(u_c)$ of a given discrete weak K.A.M. solution $u_c$  and the arithmetic property of the associated rotation number.

If $\operatorname{Sing}(u_c)\ne\emptyset$, one can pick 
  $x_0\in\operatorname{Sing}(u_c)$ and define iteratively $x_{n+1}:=\Sigma_+(x_{n}), \ \forall n\in\mathbb{N}$. In addition, if  $\operatorname{Sing}(u_c)/\mathbb{Z}$ is a finite set, we know that there exist two different natural numbers $n<m$ such that $x_m-x_n\in\mathbb{Z}$. Then we obtain that $\rho_c=(x_m-x_n)/(m-n)\in\mathbb{Q}$. Hence $\alpha'(c)$ also belongs to $\mathbb{Q}$.
  
  On the other hand, if $\alpha'(c) \in\mathbb{R}\setminus\mathbb{Q}$, due to the fact $\operatorname{Sing}(u_c)/\mathbb{Z}$ is at most countable (see \cite[Proposition 4.1.3]{Cannarsabook}), we conclude
  
\begin{corollary}
Let $f$ be a twist map of the annulus and $c\in \mathbb{R}$ be some cohomology class of $\mathbb{T}$. Let $u_c: \mathbb{R}\rightarrow\mathbb{R}$ be a discrete weak K.A.M. solution for $f$ at cohomology class $c$.
  \begin{itemize}
  \item [(1)] If $\operatorname{Sing}(u_c)/\mathbb{Z}$ is not empty and has finite elements, then $\alpha'(c)\in\mathbb{Q}$.
  \smallskip
  \item [(2)]  If $\alpha'(c) \in\mathbb{R}\setminus\mathbb{Q}$, then $\operatorname{Sing}(u_c)/\mathbb{Z}$ is either empty or infinitely countable. 
  \end{itemize}
\end{corollary}

 {
 We can use the singular dynamics $\Sigma_+$ to present a discrete analogue of Patrick Bernard’s regularization theorem \cite{Ber07}. This extends the results of \cite{zavidovique2023}.
\begin{theorem}\label{thm:regularization}
	For any bounded function $u:\mathbb{T}\rightarrow\mathbb{R}$, the function $T^{+}\circ T^{-2}[u]$ is $C^{1,1}$.
\end{theorem}
 }

In the next, we will use the singular dynamics $\Sigma_+$
to give a detailed explanation of Arnaud's observation for the connection between pseudo-graph and the Lax-Oleinik semi-group \cite{Arnaud_2011}
 in the case of twist maps (see also \cite{zavidovique2023}).

\begin{theorem}\label{thm:main2}
  Let $f$ be a twist map of the annulus  and $c\in \mathbb{R}$ be some cohomology class of $\mathbb{T}$. Let $u_c: \mathbb{R}\rightarrow\mathbb{R}$ be a discrete weak K.A.M. solution for $f$ at cohomology class $c$. Then the graphs $\operatorname{Graph}(\Sigma_+),\operatorname{Graph}(c+\d T^+[u_c])$ and $\operatorname{Graph}(c+\nabla^+u_c)$ are Lipschitz submanifolds of $\mathbb{R}^2$ associated by the commutative diagram:\begin{center}
    \begin{tikzcd}
      & \operatorname{Graph}(\Sigma_+) \arrow[rdd, "\Gamma^+"] \arrow[ldd, "\Gamma^-"'] &                                 \\
      &                                                                                 &                                 \\
    {\operatorname{Graph}(c+\d T^+[u_c])} \arrow[rr, "F"] &                                                                                 & \operatorname{Graph}(c+\nabla^+u_c)
    \end{tikzcd}
  \end{center}
  where $\Gamma^+(x,y)=(y,\partial_2S(x,y))$, $\Gamma^-(x,y)=(x,-\partial_1S(x,y))$ and $F$ is the lift of $f$ associated to $S$.
\end{theorem}

\begin{figure}[H]
  \centering

  \tikzset{every picture/.style={line width=0.75pt}} 

  \begin{tikzpicture}[x=0.75pt,y=0.75pt,yscale=-0.8,xscale=0.8]
  
  \draw [color={rgb, 255:red, 245; green, 166; blue, 35 }  ,draw opacity=1 ][line width=2.25]    (547.4,51) -- (547.4,163.6) ;
  \draw [color={rgb, 255:red, 74; green, 144; blue, 226 }  ,draw opacity=1 ][line width=2.25]    (547.4,163.6) .. controls (578.2,157.2) and (601.2,137.2) .. (610,108.2) ;
  \draw [color={rgb, 255:red, 245; green, 166; blue, 35 }  ,draw opacity=1 ][line width=2.25]    (77,114.53) -- (123,114.53) ;
  \draw [color={rgb, 255:red, 74; green, 144; blue, 226 }  ,draw opacity=1 ][line width=2.25]    (123,114.53) .. controls (154.2,91.2) and (154,62.53) .. (167,47.53) ;
  \draw    (485,33.4) -- (485,183) ;
  \draw    (610,33.4) -- (610,183) ;
  \draw    (266,33.4) -- (266,183) ;
  \draw    (391,33.4) -- (391,183) ;
  \draw    (33,46.4) -- (33,181.53) ;
  \draw    (167,47.53) -- (167,181.53) ;
  \draw    (33.95,181.53) -- (167,181.53) ;
  \draw    (32.85,47.53) -- (167,47.53) ;
  \draw [color={rgb, 255:red, 245; green, 166; blue, 35 }  ,draw opacity=1 ][line width=2.25]    (304.45,59.4) .. controls (337.6,95) and (324.6,129.4) .. (350.6,155) ;
  \draw  [dash pattern={on 0.84pt off 2.51pt}]  (33,181.53) -- (167,47.53) ;
  \draw [color={rgb, 255:red, 74; green, 144; blue, 226 }  ,draw opacity=1 ][line width=2.25]    (485,108.2) .. controls (495.2,79.2) and (523.2,52.2) .. (547.6,52.8) ;
  \draw [color={rgb, 255:red, 74; green, 144; blue, 226 }  ,draw opacity=1 ][line width=2.25]    (350.6,155) .. controls (373.2,142.2) and (384.2,131.2) .. (391,108.2) ;
  \draw [color={rgb, 255:red, 74; green, 144; blue, 226 }  ,draw opacity=1 ][line width=2.25]    (266,108.2) .. controls (275.2,81.2) and (292.2,66.2) .. (304.45,59.4) ;
  \draw [color={rgb, 255:red, 74; green, 144; blue, 226 }  ,draw opacity=1 ][line width=2.25]    (78.09,114.62) .. controls (42.2,134.2) and (46.98,166.56) .. (33.95,181.53) ;
  
  \draw (88,26.4) node [anchor=north west][inner sep=0.75pt]    {$\Sigma _{+}$};
  \draw (284,20.4) node [anchor=north west][inner sep=0.75pt]    {$c+\d T^{+}[ u_c]$};
  \draw (514,19.4) node [anchor=north west][inner sep=0.75pt]    {$c+\nabla ^{+} u_c$};

\end{tikzpicture}
  \caption{The graphs $\operatorname{Graph}(\Sigma_+),\operatorname{Graph}(c+\d T^+[u_c])$ and $\operatorname{Graph}(c+\nabla^+u_c)$  at cohomology class $c=0$ in the pendulum case.}
\end{figure}
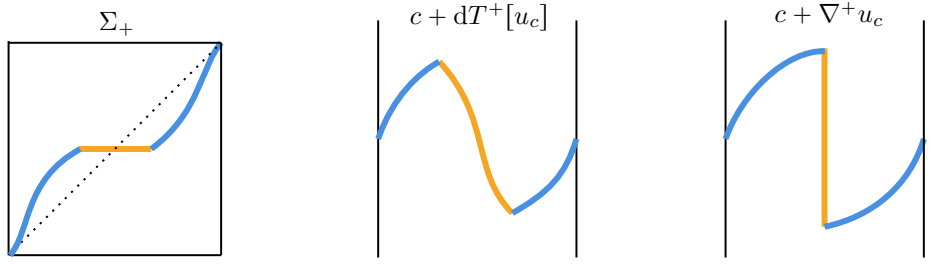

Following \cite{AZ23},
we call $\operatorname{Graph}(c+\nabla^+u_c)$ the \emph{full pseudo-graph} of a weak K.A.M. solution $u_c$ at cohomology class $c$.
If $x\in\operatorname{Sing}(u_c)$, then $\nabla^+u_c(x)$ is not a singleton and we refer to $\{x\}\times( c+\nabla^+u_c(x))$ as a vertical segment of the full pseudo-graph  $\operatorname{Graph}(c+\nabla^+u_c)$. Notice that there is a natural forward map\begin{align*}
  \Sigma_+\times \Sigma_+:\operatorname{Graph}(\Sigma_+)&\rightarrow\operatorname{Graph}(\Sigma_+) \\ 
  (x,\Sigma_+(x))&\mapsto(\Sigma_+(x),\Sigma_+^2(x)).
\end{align*} Hence for any cohomology class $c$, there are two maps \begin{align*}
  \Omega_c: \operatorname{Graph}(c+\d T^+[u_c])&\rightarrow\operatorname{Graph}(c+\d T^+[u_c]),\\ 
  \Lambda_c: \operatorname{Graph}(c+\nabla^+u_c)&\rightarrow\operatorname{Graph}(c+\nabla^+u_c)
\end{align*}
 induced by the diagram:
\begin{center}
  \begin{tikzcd}
    {\operatorname{Graph}(c+\d T^+[u_c])} \arrow[dd, "F"'] \arrow[rrrr, "\Omega_c"] &                                                                                                                                                 &  &                                                                               & {\operatorname{Graph}(c+\d T^+[u_c])} \arrow[dd, "F"] \\
                                                                         & \operatorname{Graph}(\Sigma_+) \arrow[ld, "\Gamma^+"] \arrow[lu, "\Gamma^-"'] \arrow[rr, "{\Sigma_+\times \Sigma_+}"] &  & \operatorname{Graph}(\Sigma_+) \arrow[ru, "\Gamma^-"] \arrow[rd, "\Gamma^+"'] &                                                   \\
    \operatorname{Graph}(c+\nabla^+u_c) \arrow[rrrr, "\Lambda_c"]                    &                                                                                                                                                 &  &                                                                               & \operatorname{Graph}(c+\nabla^+u_c)                  
    \end{tikzcd}
\end{center}
\medskip

  The dynamical systems defined by $\Lambda_c$ and $\Omega_c$ are conjugate and also termed as the singular dynamics on $\operatorname{Graph}(c+\nabla^+u_c)$ and $\operatorname{Graph} (c+\d T^{+}[ u_c])$  respectively (see Figures~\ref{fig:SDond+u} and  \ref{fig:SDondT+u} for an illustration).

  \begin{figure}[H]
  \vspace{2.5em}

    \centering

    \tikzset{every picture/.style={line width=0.75pt}} 
    
    \begin{tikzpicture}[x=0.75pt,y=0.75pt,yscale=-0.8,xscale=0.8]
    
    \draw    (57.4,278.2) -- (619.4,278.2) ;
    \draw [color={rgb, 255:red, 144; green, 19; blue, 254 }  ,draw opacity=1 ][line width=2.25]    (392.4,72.2) -- (392.4,85.2) ;
    \draw [color={rgb, 255:red, 245; green, 166; blue, 35 }  ,draw opacity=1 ][line width=2.25]    (392.4,89.4) -- (392.4,147.53) ;
    \draw [color={rgb, 255:red, 80; green, 227; blue, 194 }  ,draw opacity=1 ][line width=2.25]    (392.4,147.53) -- (392.4,185.2) ;
    \draw [line width=2.25]    (365.8,102.2) .. controls (368.4,88.2) and (379.4,74.2) .. (392.4,72.2) ;
    \draw [color={rgb, 255:red, 74; green, 144; blue, 226 }  ,draw opacity=1 ][line width=2.25]    (392.4,185.2) .. controls (402.8,185.4) and (437.4,174.2) .. (444.4,137.2) ;
    \draw    (101,35.2) -- (101,329.2) ;
    \draw  [dash pattern={on 0.84pt off 2.51pt}]  (392.4,185.2) -- (392.4,278.2) ;
    \draw  [color={rgb, 255:red, 208; green, 2; blue, 27 }  ,draw opacity=1 ][fill={rgb, 255:red, 208; green, 2; blue, 27 }  ,fill opacity=1 ] (388.2,89.4) .. controls (388.2,87.08) and (390.08,85.2) .. (392.4,85.2) .. controls (394.72,85.2) and (396.6,87.08) .. (396.6,89.4) .. controls (396.6,91.72) and (394.72,93.6) .. (392.4,93.6) .. controls (390.08,93.6) and (388.2,91.72) .. (388.2,89.4) -- cycle ;
    \draw  [color={rgb, 255:red, 126; green, 211; blue, 33 }  ,draw opacity=1 ][fill={rgb, 255:red, 126; green, 211; blue, 33 }  ,fill opacity=1 ] (388.2,147.53) .. controls (388.2,145.21) and (390.08,143.33) .. (392.4,143.33) .. controls (394.72,143.33) and (396.6,145.21) .. (396.6,147.53) .. controls (396.6,149.85) and (394.72,151.73) .. (392.4,151.73) .. controls (390.08,151.73) and (388.2,149.85) .. (388.2,147.53) -- cycle ;
    \draw [color={rgb, 255:red, 0; green, 0; blue, 0 }  ,draw opacity=1 ][line width=2.25]    (529.4,72.2) -- (529.4,100.2) ;
    \draw [color={rgb, 255:red, 74; green, 144; blue, 226 }  ,draw opacity=1 ][line width=2.25]    (529.4,96) -- (529.4,157.2) ;
    \draw  [dash pattern={on 0.84pt off 2.51pt}]  (529.4,157.2) -- (529.4,278.2) ;
    \draw  [color={rgb, 255:red, 0; green, 0; blue, 0 }  ,draw opacity=1 ][fill={rgb, 255:red, 255; green, 255; blue, 255 }  ,fill opacity=1 ] (525.2,96) .. controls (525.2,93.68) and (527.08,91.8) .. (529.4,91.8) .. controls (531.72,91.8) and (533.6,93.68) .. (533.6,96) .. controls (533.6,98.32) and (531.72,100.2) .. (529.4,100.2) .. controls (527.08,100.2) and (525.2,98.32) .. (525.2,96) -- cycle ;
    \draw  [dash pattern={on 0.84pt off 2.51pt}]  (365.8,102.2) -- (365.8,278.2) ;
    \draw  [dash pattern={on 0.84pt off 2.51pt}]  (444.4,137.2) -- (444.4,278.2) ;
    \draw    (423.4,52.2) .. controls (456.72,15.94) and (506.37,27.7) .. (522.46,55.48) ;
    \draw [shift={(523.4,57.2)}, rotate = 242.65] [color={rgb, 255:red, 0; green, 0; blue, 0 }  ][line width=0.75]    (10.93,-3.29) .. controls (6.95,-1.4) and (3.31,-0.3) .. (0,0) .. controls (3.31,0.3) and (6.95,1.4) .. (10.93,3.29)   ;
    \draw [color={rgb, 255:red, 126; green, 211; blue, 33 }  ,draw opacity=1 ][fill={rgb, 255:red, 126; green, 211; blue, 33 }  ,fill opacity=1 ][line width=2.25]    (234.4,92) -- (234.4,223.67) ;
    \draw [color={rgb, 255:red, 208; green, 2; blue, 27 }  ,draw opacity=1 ][line width=2.25]    (178,104) -- (178,135.8) ;
    \draw [color={rgb, 255:red, 245; green, 166; blue, 35 }  ,draw opacity=1 ][line width=2.25]    (178,135.8) .. controls (206.4,138) and (203.4,94) .. (234.4,92) ;
    \draw [color={rgb, 255:red, 144; green, 19; blue, 254 }  ,draw opacity=1 ][line width=2.25]    (151.4,134) .. controls (160.4,119) and (165,106) .. (178,104) ;
    \draw [color={rgb, 255:red, 80; green, 227; blue, 194 }  ,draw opacity=1 ][line width=2.25]    (234.4,223.67) .. controls (244.8,223.87) and (264.8,197.67) .. (268.8,187.67) ;
    \draw  [dash pattern={on 0.84pt off 2.51pt}]  (151.4,134) -- (151.4,275.4) ;
    \draw  [dash pattern={on 0.84pt off 2.51pt}]  (268.8,187.67) -- (268.8,276.73) ;
    \draw [color={rgb, 255:red, 74; green, 144; blue, 226 }  ,draw opacity=1 ][line width=2.25]    (268.8,187.67) .. controls (272.8,178.67) and (283.8,182.67) .. (289.8,172.67) ;
    \draw [line width=2.25]    (131.4,157) .. controls (134,142) and (148.4,142) .. (151.4,134) ;
    \draw    (264.73,70.87) .. controls (298.05,34.61) and (347.7,46.37) .. (363.79,74.15) ;
    \draw [shift={(364.73,75.87)}, rotate = 242.65] [color={rgb, 255:red, 0; green, 0; blue, 0 }  ][line width=0.75]    (10.93,-3.29) .. controls (6.95,-1.4) and (3.31,-0.3) .. (0,0) .. controls (3.31,0.3) and (6.95,1.4) .. (10.93,3.29)   ;
    
    \draw (394.4,281.6) node [anchor=north west][inner sep=0.75pt]    {$y$};
    \draw (137.4,279.6) node [anchor=north west][inner sep=0.75pt]    {$x_{-}( y)$};
    \draw (258,280.4) node [anchor=north west][inner sep=0.75pt]    {$x_{+}( y)$};
    \draw (510,281.4) node [anchor=north west][inner sep=0.75pt]    {$\Sigma _{+}( y)$};
    \draw (116,13.4) node [anchor=north west][inner sep=0.75pt]    {$\operatorname{Graph(} c+\nabla ^{+} u_c)$};
    \draw (317.33,281.4) node [anchor=north west][inner sep=0.75pt]  [font=\normalsize]  {$x_{-}( \Sigma _{+}( y))$};
    \draw (411.67,280.73) node [anchor=north west][inner sep=0.75pt]  [font=\normalsize]  {$x_{+}( \Sigma _{+}( y))$};
    \draw (465,11.4) node [anchor=north west][inner sep=0.75pt]    {$\Lambda_c$};
    \draw (306.33,30.07) node [anchor=north west][inner sep=0.75pt]    {$\Lambda_c$};

    \end{tikzpicture}
    \caption{Singular dynamics on the full pseudo-graph $\operatorname{Graph}(c+\nabla^+u_c)$ of a discrete weak K.A.M. solution $u_c$ at cohomology class $c$, where different colors are used to represent the corresponding pieces before and after the map $\Lambda_c$.}
    \label{fig:SDond+u}
    \end{figure}
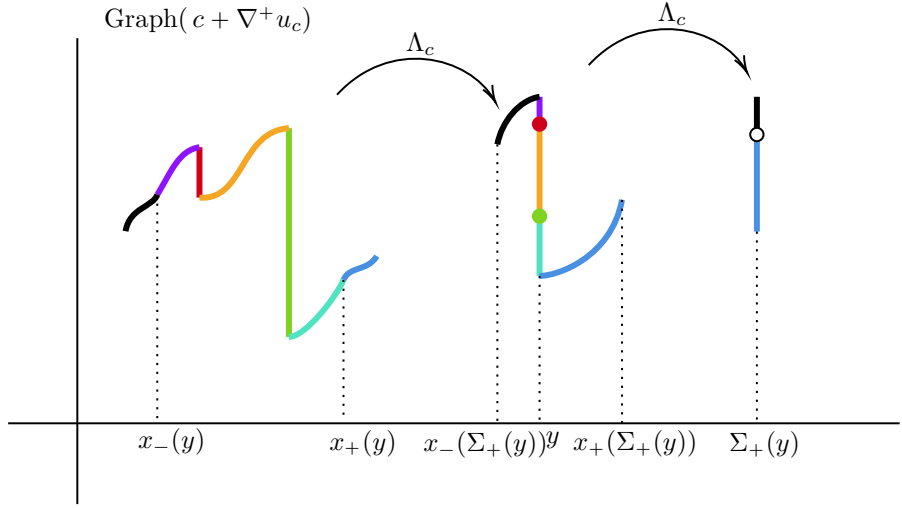

    \begin{figure}[H]
      
      \centering

      \tikzset{every picture/.style={line width=0.75pt}} 
    
      \begin{tikzpicture}[x=0.75pt,y=0.75pt,yscale=-0.8,xscale=0.8]
      
      \draw [color={rgb, 255:red, 245; green, 166; blue, 35 }  ,draw opacity=1 ][line width=2.25]    (407.77,130.51) .. controls (422.6,137) and (427.6,161) .. (433.6,165) .. controls (439.6,169) and (443.6,142) .. (463.13,173) ;
      \draw [line width=2.25]    (549.4,142.47) .. controls (568.4,138.47) and (566.4,160.93) .. (576.4,166.13) ;
      \draw [color={rgb, 255:red, 74; green, 144; blue, 226 }  ,draw opacity=1 ][line width=2.25]    (576.4,166.13) .. controls (595.4,197.73) and (618.4,143.87) .. (628.4,180.67) ;
      \draw  [dash pattern={on 0.84pt off 2.51pt}]  (576.4,166) -- (576.4,291.6) ;
      \draw    (55.4,291) -- (617.4,291) ;
      \draw    (99,48) -- (99,342) ;
      \draw    (444.4,112) .. controls (477.72,75.74) and (527.37,87.5) .. (543.46,115.28) ;
      \draw [shift={(544.4,117)}, rotate = 242.65] [color={rgb, 255:red, 0; green, 0; blue, 0 }  ][line width=0.75]    (10.93,-3.29) .. controls (6.95,-1.4) and (3.31,-0.3) .. (0,0) .. controls (3.31,0.3) and (6.95,1.4) .. (10.93,3.29)   ;
      \draw  [color={rgb, 255:red, 0; green, 0; blue, 0 }  ,draw opacity=1 ][fill={rgb, 255:red, 255; green, 255; blue, 255 }  ,fill opacity=1 ] (572.2,166) .. controls (572.2,163.68) and (574.08,161.8) .. (576.4,161.8) .. controls (578.72,161.8) and (580.6,163.68) .. (580.6,166) .. controls (580.6,168.32) and (578.72,170.2) .. (576.4,170.2) .. controls (574.08,170.2) and (572.2,168.32) .. (572.2,166) -- cycle ;
      \draw [color={rgb, 255:red, 144; green, 19; blue, 254 }  ,draw opacity=1 ][line width=2.25]    (131.4,202.4) .. controls (140.4,198.4) and (142.4,163.4) .. (163.4,176.8) ;
      \draw [color={rgb, 255:red, 245; green, 166; blue, 35 }  ,draw opacity=1 ][line width=2.25]    (182.13,204.8) .. controls (207.47,211.47) and (201.13,157.8) .. (219.13,156.8) ;
      \draw [color={rgb, 255:red, 80; green, 227; blue, 194 }  ,draw opacity=1 ][line width=2.25]    (261.4,218.8) .. controls (275.4,231.8) and (284.4,231.8) .. (302.4,215.8) ;
      \draw [line width=2.25]    (107.01,221.4) .. controls (122.4,212.4) and (117.4,205.4) .. (131.4,202.4) ;
      \draw [color={rgb, 255:red, 74; green, 144; blue, 226 }  ,draw opacity=1 ][line width=2.25]    (302.4,215.8) .. controls (315.4,189.8) and (328.4,196.8) .. (333.4,181) ;
      \draw [color={rgb, 255:red, 208; green, 2; blue, 27 }  ,draw opacity=1 ][line width=2.25]    (163.4,176.8) .. controls (165.4,192) and (170.13,200.8) .. (182.13,204.8) ;
      \draw [color={rgb, 255:red, 126; green, 211; blue, 33 }  ,draw opacity=1 ][line width=2.25]    (219.13,156.8) .. controls (230.13,153.8) and (252.4,208.8) .. (261.4,218.8) ;
      \draw  [dash pattern={on 0.84pt off 2.51pt}]  (131.4,202.4) -- (131.4,292.53) ;
      \draw  [dash pattern={on 0.84pt off 2.51pt}]  (302.4,215.8) -- (302.4,291.2) ;
      \draw   (132,290) .. controls (132,294.67) and (134.33,297) .. (139,297) -- (207.07,297) .. controls (213.74,297) and (217.07,299.33) .. (217.07,304) .. controls (217.07,299.33) and (220.4,297) .. (227.07,297)(224.07,297) -- (295.13,297) .. controls (299.8,297) and (302.13,294.67) .. (302.13,290) ;
      \draw    (247.47,134.4) .. controls (255.39,106.68) and (301.2,84.19) .. (348.31,113.43) ;
      \draw [shift={(349.73,114.33)}, rotate = 212.74] [color={rgb, 255:red, 0; green, 0; blue, 0 }  ][line width=0.75]    (10.93,-3.29) .. controls (6.95,-1.4) and (3.31,-0.3) .. (0,0) .. controls (3.31,0.3) and (6.95,1.4) .. (10.93,3.29)   ;
      \draw [color={rgb, 255:red, 80; green, 227; blue, 194 }  ,draw opacity=1 ][line width=2.25]    (465.24,175.28) .. controls (477.24,191.54) and (479.87,208.6) .. (498.13,190.13) ;
      \draw  [color={rgb, 255:red, 126; green, 211; blue, 33 }  ,draw opacity=1 ][fill={rgb, 255:red, 126; green, 211; blue, 33 }  ,fill opacity=1 ] (461.43,177.04) .. controls (460.45,174.94) and (461.36,172.44) .. (463.47,171.47) .. controls (465.57,170.49) and (468.07,171.4) .. (469.04,173.51) .. controls (470.02,175.61) and (469.11,178.11) .. (467,179.09) .. controls (464.9,180.06) and (462.4,179.15) .. (461.43,177.04) -- cycle ;
      \draw [line width=2.25]    (360.46,160.07) .. controls (363.06,145.07) and (377.46,145.07) .. (380.46,137.07) ;
      \draw [color={rgb, 255:red, 74; green, 144; blue, 226 }  ,draw opacity=1 ][line width=2.25]    (498.13,190.13) .. controls (502.13,181.13) and (513.13,185.13) .. (519.13,175.13) ;
      \draw  [dash pattern={on 0.84pt off 2.51pt}]  (380.46,137.07) -- (380.46,290.13) ;
      \draw  [dash pattern={on 0.84pt off 2.51pt}]  (498.13,190.13) -- (498.13,290.13) ;
      \draw [color={rgb, 255:red, 144; green, 19; blue, 254 }  ,draw opacity=1 ][line width=2.25]    (380.12,138.6) .. controls (394.6,125.2) and (390.6,118) .. (407.77,130.51) ;
      \draw  [color={rgb, 255:red, 208; green, 2; blue, 27 }  ,draw opacity=1 ][fill={rgb, 255:red, 208; green, 2; blue, 27 }  ,fill opacity=1 ] (402.96,131.28) .. controls (401.98,129.17) and (402.9,126.68) .. (405,125.7) .. controls (407.1,124.72) and (409.6,125.64) .. (410.58,127.74) .. controls (411.55,129.85) and (410.64,132.34) .. (408.54,133.32) .. controls (406.43,134.3) and (403.94,133.38) .. (402.96,131.28) -- cycle ;
      
      \draw (114,26.2) node [anchor=north west][inner sep=0.75pt]    {$\operatorname{Graph(} c+\d T^{+}[ u_c])$};
      \draw (486,71.2) node [anchor=north west][inner sep=0.75pt]    {$\Omega_c$};
      \draw (570,295.4) node [anchor=north west][inner sep=0.75pt]    {$y$};
      \draw (361.4,293.53) node [anchor=north west][inner sep=0.75pt]    {$x_{-}( y)$};
      \draw (486.4,293.8) node [anchor=north west][inner sep=0.75pt]    {$x_{+}( y)$};
      \draw (184.67,306.4) node [anchor=north west][inner sep=0.75pt]    {$( \Sigma _{+})^{-2}( y)$};
      \draw (282,79.2) node [anchor=north west][inner sep=0.75pt]    {$\Omega_c$};

      \end{tikzpicture}
      \caption{Singular dynamics on the $\operatorname{Graph} (c+\d T^{+}[ u_c])$,  where different colors are used to represent the corresponding pieces before and after the map $\Omega_c$.}
      \label{fig:SDondT+u}
    \end{figure}
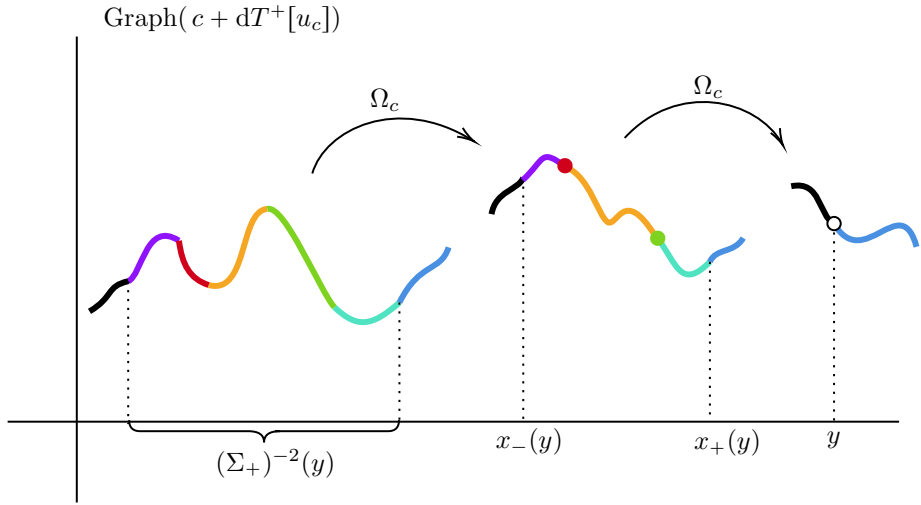

Consequently, one could analyze the dynamics $\Lambda_c$ on the full pseudo-graphs $\operatorname{Graph}(c+\nabla^+u_c)$ and has the following theorem.
\begin{theorem}\label{thm:SingularDynamics}
  Let $f$ be a twist map of the annulus  and $c\in \mathbb{R}$ be some cohomology class of $\mathbb{T}$. Let $u_c: \mathbb{R}\rightarrow\mathbb{R}$ be a discrete weak K.A.M. solution for $f$ at cohomology class $c$. Then the map   \begin{align*}\Lambda_c:\operatorname{Graph}(c+\nabla^+u_c)&\rightarrow\operatorname{Graph}(c+\nabla^+u_c)\\ 
  (x,p)&\mapsto (y,\partial_2 S(x,y))\end{align*} is well-defined and implicitly given by \begin{align*}
    T^+[u_c](x)=u_c(y)-S_c(x,y).
  \end{align*}
  Moreover,  the map $\Lambda_c$ satisfies the following properties: 
 \begin{enumerate}
    \item [(i)] $\Lambda_c(x+1,p)=\Lambda_c(x,p)+(1,0)$;
    \item [(ii)] $\Lambda_c$ is  Lipschitz continuous;
    \item [(iii)] $\Lambda_c$ preserves the orientation of the full pseudo-graph.
    \item [(iv)] the image of a vertical segment under $\Lambda_c$ is a point that belongs to a vertical segment.
    \item [(v)] $\Lambda_c$ admits a rotation number for positive iterations, that is, 
    \begin{align*}
      \rho(\Lambda_c):=\lim_{n\to+\infty}\frac{1}{n}\left( \pi_1\circ\Lambda_c^n(x,p)-x \right)
    \end{align*}
  exists for all $(x,p)\in\operatorname{Graph}(c+\nabla^+u_c)$ and is independent of $(x,p)$. In fact, $\rho(\Lambda_c) = \alpha'(c)$.
  \end{enumerate}
\end{theorem}

Heuristically, from Figure~\ref{fig:SDond+u}, one could  observe that some points are attracted to the singularities of $u_c$ under the singular dynamics $\Lambda_c$. However, on the other hand, it is also evident that some other points
   do not propagate into the singularities  during the iteration of $\Lambda_c$. Thus, one could conclude:
 {
\begin{theorem}\label{thm:main3}
  Let $f$ be a twist map of the annulus  and $c\in \mathbb{R}$ be some cohomology class of $\mathbb{T}$. For any  discrete weak K.A.M. solution $u_c: \mathbb{R}\rightarrow\mathbb{R}$ for $f$ at cohomology class $c$, the $\alpha$-limit set of $\operatorname{Graph}(c+\d u_c)$ under $\Lambda_c$: $$\alpha(\operatorname{Graph}(c+\d u_c)):=\bigcap_{n\in\mathbb{Z}_+}\overline{\{\Lambda^i_c(x,p)\mid i\leq -n, (x,p)\in \operatorname{Graph}(c+\d u_c)\}}$$
  is closed, non-empty and invariant under the dynamics $F$. Furthermore, the Aubry set at cohomology class $c$ can be expressed as $$\mathcal{A}^*_c=\bigcap_{u_c}\alpha(\operatorname{Graph}(c+\d u_c))-(0,c),$$ where the intersection is taken over all discrete weak K.A.M. solutions $u_c$ at  cohomology class $c$. 
  \end{theorem}
  
  We can extend some classical results \cite[Proposition 11.2.5]{katok1995introduction} for orientation-preserving  circle homeomorphism to orientation-preserving continuous  circle map of degree $1$, without requiring invertibility. Note that $\Sigma_+$ is precisely such a map. Consequently, if  $\alpha'(c) \in\mathbb{R}\setminus\mathbb{Q}$, the $\alpha$-limit set under $\Sigma_+$ is  either $\mathbb{R}$ or a Cantor set. The same conclusion holds for the $\alpha$-limit set under $\Lambda_c$, since $\Sigma_+\circ \pi_1=\pi_1\circ \Lambda_c$. Furthermore, a theorem due to Mather \cite{mather1990differentiability} and Bangert \cite{bangert1988mather} guarantees the uniqueness of discrete weak K.A.M. solutions $u_c$, up to constants, when $\alpha'(c) \in\mathbb{R}\setminus\mathbb{Q}$. Thus, we conclude:
  \begin{corollary}
Let $f$ be a twist map of the annulus and $c\in \mathbb{R}$ be some cohomology class of $\mathbb{T}$. If $\alpha'(c) \in\mathbb{R}\setminus\mathbb{Q}$, then one of the following holds:
  \begin{itemize}
  \item [(1)]  $\mathcal{A}^*_c$ is a Lipschitz invariant circle.
  \item [(2)]  $\mathcal{A}^*_c$ is a Lipschitz Cantori.
    \end{itemize}
\end{corollary}

 }
 
 \subsection{Notations and Hypotheses}
 \begin{notation}
  \begin{itemize}
      \item $\mathbb{T}^d=\mathbb{R}^d/\mathbb{Z}^d$ is the $d$-dimensional torus and its cotangent bundle $T^*\mathbb{T}^d=\mathbb{T}^d\times (\mathbb{R}^d)^*$ is the $2d$-dimensional annulus. The points of the annulus are denoted by $(\theta,p)\in \mathbb{T}^d\times (\mathbb{R}^d)^*$ and the points of its universal cover are denoted by $(x,p)\in\mathbb{R}^d\times (\mathbb{R}^d)^*$.
      
      \item Let $S(x_1, \dots, x_{2d})$ be a $C^1$ real-valued function defined on $\mathbb{R}^{2d}$. We denote by $\partial_1 S$ and $\partial_2 S$ the $(\mathbb{R}^d)^*$-valued functions $(\partial S/\partial x_1, \dots, \partial S/\partial x_d)$ and $(\partial S/\partial x_{d+1}, \dots, \partial S/\partial x_{2d})$  respectively.
      \item   For the product space $X_1\times \cdots\times X_n$ and $k\in\{1,\dots,n\}$, the map $\pi_k:X_1\times \cdots\times X_n\rightarrow X_k$ is the projection on the $k$-th variable.
      \item Let $\mathcal{G}:X\rightarrow 2^{Y}$ be a set-valued map, we define its graph as a subset of $X\times Y$: \begin{align*}
        \operatorname{Graph}(\mathcal{G}):=\{(x,y)\in X\times Y\mid y\in\mathcal{G}(x)\}.
      \end{align*} 
      For example, if  $u:\mathbb{T}^d\rightarrow\mathbb{R}$ is semiconcave, we will denote by $\operatorname{Graph}(\d u)$ the partial graph of $\d u$:\begin{align*}
        \operatorname{Graph}(\d u)=\{(x,y)\in\mathbb{R}^d\times (\mathbb{R}^d)^*\mid u \text{ is differentiable at $x$ and }y=\d u(x)\}.
      \end{align*}
      And we will denote by $\operatorname{Graph}(\nabla^+ u)$ the full pseudo-graph of $u$:\begin{align*}
        \operatorname{Graph}(\nabla^+ u)=\{(x,y)\in\mathbb{R}^d\times (\mathbb{R}^d)^*\mid y\in \nabla^+u(x)\}.
      \end{align*}
      \item We denote by $B(\mathbb{T}^d)$, $C^0(\mathbb{T}^d)$ and $SC(\mathbb{T}^d)$
        the spaces of bounded $\mathbb{Z}^d$-periodic functions, continuous $\mathbb{Z}^d$-periodic functions and semiconcave $\mathbb{Z}^d$-periodic functions. 
        \item A mapping is of class $C^{1,1}$ if it is differentiable with uniformly Lipschitz derivative.       
  \end{itemize}
  \end{notation}    
  
For the sake of slight generality, we use

\begin{hypothesis}\label{hypothesis:S1}
  The function $S:\mathbb{R}^{d}\times\mathbb{R}^d\rightarrow\mathbb{R}$ satisfies the following properties:  \begin{enumerate}
  	\item $S$ is continuous;
    \item  $S$ is coercive: $$\lim_{\|x-y\|\rightarrow+\infty} S(x,y)=+\infty;$$
    \item  $S$ is diagonal-periodic:\begin{align*}
      \forall (x,y)\in\mathbb{R}^{d}\times\mathbb{R}^{d},m\in\mathbb{Z}^d \quad S(x+m,y+m)=S(x,y);
    \end{align*}
    \end{enumerate}
\end{hypothesis}
\begin{hypothesis}\label{hypothesis:S2}
  The function $S:\mathbb{R}^{d}\times\mathbb{R}^d\rightarrow\mathbb{R}$ satisfies the following properties:  \begin{enumerate}
    \item  $S$ is locally semiconcave;
    \item  $S$ is ferromagnetic; that is, $S$ is of class $C^1$ and 
     the two maps $
        x \mapsto \partial_2S(x, y) $
        and $ y \mapsto \partial_1S(x, y)$
    are homeomorphisms between $\mathbb{R}^d$ and $(\mathbb{R}^d)^*$ for all $ (x, y)$.
  \end{enumerate}
\end{hypothesis}
\begin{hypothesis}\label{hypothesis:S3}
 The dimension $d=1$ and the Aubry \& Le Daeron Non-crossing Lemma \cite{aubry1983discrete} holds: if $(x_1-x_2)(y_1-y_2)<0$, then $S(x_1,y_2)+S(x_2,y_1)<S(x_1,y_1)+S(x_2,y_2)$.
\end{hypothesis}

\begin{remark}\label{remark:F}
 One can refer the notion of ferromagnetic terminology in \cite[Definition 2.4]{GT11}. If $S$ satisfies Hypothesis \ref{hypothesis:S2}, then the discrete standard map $
    F:\mathbb{R}^d\times (\mathbb{R}^d)^*\righttoleftarrow$ can be defined by \begin{align*}
    F(x,p)=(y,p') \Longleftrightarrow\begin{cases}
      p=-\partial_1S(x,y)\\ 
      p'=\partial_2S(x,y)
    \end{cases}.
  \end{align*}
  The map $F$ is a homeomorphism and can be seen as a discrete version of the Hamiltonian flow.
\end{remark}

 \subsection{Plan for this paper}
 
 {
 This paper is organized as follows. In Section~\ref{section:DWKT}, we first provide several fundamental facts about discrete weak K.A.M. theory on tori of any dimension.
 
In Section~\ref{subsection:Aubry}, we discuss an important invariant set in discrete weak K.A.M. theory—the Aubry set. In Section~\ref{subsection:Regularity}, we address some regularity results in discrete weak K.A.M. theory, particularly those related to semiconcave functions. The conclusions in Section 2.1 require Hypothesis 1, while those in Section 2.2 require both Hypothesis 1 and Hypothesis 2.

Theorems~\ref{thm:main}-\ref{thm:packaged} and Theorems~\ref{thm:regularization}-\ref{thm:main3} are proved in Sections~\ref{sec:Theorem1} and \ref{sec:consequences}, respectively. The conclusions in Sections 3 and 4 require the validity of Hypothesis 1, Hypothesis 2, and Hypothesis 3. 

In Appendix~\ref{sec:appendixA}, we list some fundamental properties of semiconcave functions without proof. Appendix~\ref{sec:appendixB} reviews key results for classical exact twist maps. The proofs of Lemmas~\ref{lemma:RegularityOfSubaction} and \ref{lemma:WellDefined} are included in Appendix~\ref{sec:appendixC} for self-containedness. 
 }

\section{Discrete weak K.A.M. theory on $d$-dimensional torus}\label{section:DWKT}
\subsection{A notion of important invariant set}\label{subsection:Aubry} As standard, we introduce the following invariant sets and investigate several properties of them. In Section \ref{subsection:Aubry} , we operate under the assumption that the function  $S$  satisfies Hypothesis \ref{hypothesis:S1}.

We recall the \emph{discrete backward Lax-Oleinik operator} and  \emph{discrete forward Lax-Oleinik operator} defined in the introduction: 

\begin{definition}
	The \emph{discrete backward Lax-Oleinik operator} $T^-$ and \emph{discrete forward Lax-Oleinik operator} $T^+$ are defined by
\begin{align*}
  \forall y\in\mathbb{R}^d,\quad T^- [u] (y) := &\inf_{x \in \mathbb{R}^d} \big\{ u(x) + S (x, y) \big\},\\ 
  \forall x\in\mathbb{R}^d,\quad T^+ [u] (x) := &\sup_{y \in \mathbb{R}^d} \big\{ u(y) - S (x, y) \big\},
\end{align*}
for every bounded $\mathbb{Z}^d$-periodic function (not necessarily continuous) $u:\mathbb{R}^d\rightarrow\mathbb{R}$ (namely $u\in B(\mathbb{T}^d)$). \end{definition}

\begin{remark}\label{remark:LaxOleinik}
  \begin{enumerate}
    \item [(a)] 
    The image $T^-(B(\mathbb{T}^d))$ of $T^-$ is equicontinuous and has uniformly bounded oscillation, as shown in \cite[Lemma 4.2]{GT11}.
    \item [(b)] We can derive directly from the definition that \begin{align*}
  \forall y\in\mathbb{R}^d,\quad T^{-n} [u] (y) = &\inf_{\substack{(y_{-n},\dots,y_{-1}) \in (\mathbb{R}^d)^n\\ y_0=y}} \left\{ v(y_{-n}) + \sum_{i={-n}}^{-1}S (y_{i}, y_{i+1}) \right\},\\ 
  \forall x\in\mathbb{R}^d,\quad T^{+n} [u] (x) = &\sup_{\substack{(x_1,\dots,x_n) \in (\mathbb{R}^d)^n\\ x_0=x}} \left\{ u(x_n) - \sum_{i=0}^{n-1}S (x_i, x_{i+1}) \right\}.
\end{align*} 
  \end{enumerate}
\end{remark}

\begin{definition}Let $u\in C^0(\mathbb{T}^d)$ and $n\in\mathbb{Z}_+$. \begin{enumerate}
\item The \emph{optimal backward map} is the closed multi-valued map \begin{align*}
    \Sigma_-[u]:\mathbb{R}^d&\rightarrow 2^{\mathbb{R}^d}\\ 
    y&\mapsto \mathop{\arg\min}_{x\in\mathbb{R}^d}\{u(x)+S(x,y)\};
  \end{align*}
  \item the \emph{optimal forward map} is the closed multi-valued map \begin{align*}
    \Sigma_+[u]:\mathbb{R}^d&\rightarrow2^{\mathbb{R}^d}\\ 
    x&\mapsto \mathop{\arg\max}_{y\in\mathbb{R}^d}\{u(y)-S(x,y)\};
  \end{align*}
  \item The \emph{  $n$-step optimal backward map} is the multi-valued map \begin{align*}
    \Sigma_-^n[u]:\mathbb{R}^d&\rightarrow2^{\mathbb{R}^d}\\ 
    y&\mapsto \pi_1\left( \mathop{\arg\min}_{\substack{(y_{-n},\dots,y_{-1}) \in (\mathbb{R}^d)^n\\ y_0=y}}\left\{ u(y_{-n}) + \sum_{i={-n}}^{-1}S (y_{i}, y_{i+1}) \right\}\right);
  \end{align*}
  \item the \emph{ $n$-step optimal forward map} is the multi-valued map \begin{align*}
    \Sigma_+^n[u]:\mathbb{R}^d&\rightarrow2^{\mathbb{R}^d}\\ 
    x&\mapsto \pi_n\left(\mathop{\arg\max}_{\substack{(x_1,\dots,x_n) \in (\mathbb{R}^d)^n\\ x_0=x}}\left\{ u(x_n) - \sum_{i=0}^{n-1}S (x_i, x_{i+1}) \right\}\right).
  \end{align*}
\end{enumerate} 
  
\end{definition}

\begin{remark}\label{remark:coercive}
\begin{enumerate}
	\item [(a)] Given $u\in C^0(\mathbb{T}^d)$, as $S$ is continuous and coercive (Hypothesis \ref{hypothesis:S1}), for any $y\in\mathbb{R}^d$, we can find $x\in\mathbb{R}^d$ that attains the infimum in the definition of $T^-[u](y)$. Similarly, for any $x\in\mathbb{R}^d$, we can find $y\in\mathbb{R}^d$ that attains the supremum in the definition of $T^+[u](x)$. In other words, $\Sigma_-[u](y)\ne\emptyset, \forall y\in\mathbb{R}^d$ and $\Sigma_+[u](x)\ne\emptyset, \forall x\in\mathbb{R}^d$. 
	\item [(b)] We can derive directly from the definition and Remark \ref{remark:LaxOleinik} (b) that for any $u\in C^0(\mathbb{T}^d)$, $m\in\mathbb{Z}_+$, $n\in\mathbb{Z}_+$ and $y\in\mathbb{R}^d$, $$\Sigma_-^m[u]\circ\Sigma_-^n[T^{-m}[u]](y)=\Sigma_-^{m+n}[u](y).$$
\end{enumerate}
	\end{remark}

 \begin{lemma}\label{lemma:Sigma-andSigma+}
 	Let $u=T^{-n}[v]$, $v\in C^0(\mathbb{T}^d)$ and $x,y\in\mathbb{R}^d$. The following two statements are equivalent: \begin{enumerate}
 		\item [(i)]$x\in\Sigma_-^n[v](y)$;
 		\item [(ii)]$y\in \Sigma_+^n[u](x)$ and $T^{+n}\circ T^{-n}[v](x)=v(x).$
 	\end{enumerate} \end{lemma}
 
 \begin{proof}
 	(i) $\Rightarrow$ (ii): Suppose that $x\in\Sigma_-^n[v](y)$.	Using the definition of $T^-$, by $u=T^{-n}[v]$, we have \begin{equation}\label{eq:Sigma-andSigma+}
 		u(y)=\inf_{\substack{(y_{-n},\dots,y_{-1}) \in (\mathbb{R}^d)^n\\ y_0=y}} \left\{ v(y_{-n}) + \sum_{i={-n}}^{-1}S (y_{i}, y_{i+1}) \right\}=v(x)+\sum_{i={-n}}^{-1}S(y_i,y_{i+1})
 	\end{equation} where $y_{-n}=x,y_0=y$ and for any $(z_{-n+1},\dots,z_{0}=z) \in (\mathbb{R}^d)^n$, $$u(z)=\inf_{\substack{(z_{-n},\dots,z_{-1}) \in (\mathbb{R}^d)^n\\ z_0=z}} \left\{ v(z_{-n}) + \sum_{i={-n}}^{-1}S (z_{i}, z_{i+1}) \right\}\leq v(x)+\sum_{i={-n}}^{-1}S(z_i,z_{i+1}),$$
 	where $z_{-n}=x, z_{0}=z$.
	By combining these equation and inequality and eliminating $v(x)$, we obtain $$u(y)-\sum_{i={-n}}^{-1}S(y_i,y_{i+1})\geq u(z)-\sum_{i={-n}}^{-1}S(z_i,z_{i+1}),\quad \forall (z_{-n+1},\dots,z_{0}=z) \in (\mathbb{R}^d)^n,$$ which implies  $T^{+n}[u](x)=u(y)-\sum_{i={-n}}^{-1}S(y_i,y_{i+1})$ and $y\in \Sigma_+^n[u](x)$. Notice that the equation \eqref{eq:Sigma-andSigma+} says $u(y)-\sum_{i={-n}}^{-1}S(y_i,y_{i+1})=v(x)$. Hence $T^{+n}[u](x)=v(x)$.
	
	(ii) $\Rightarrow$ (i): Suppose that $y\in \Sigma_+^n[u](x)$ and $T^{+n}\circ T^{-n}[v](x)=v(x).$ Using the definition of $T^-$ and $T^+$, by $u=T^{-n}[v]$, we have $$
 		u(y)\leq v(y_{-n}) + \sum_{i={-n}}^{-1}S (y_{i}, y_{i+1}) 		$$ for any $(y_{-n},\dots,y_{-1}) \in (\mathbb{R}^d)^n$, $y_0=y$ and  $$T^{+n}[u](x)=\sup_{\substack{(x_1,\dots,x_n) \in (\mathbb{R}^d)^n\\ x_0=x}} \left\{ u(x_n) - \sum_{i=0}^{n-1}S (x_i, x_{i+1}) \right\}=u(y) - \sum_{i=0}^{n-1}S (x_i, x_{i+1})$$
 		where $x_0=x,x_n=y$.
By combining these equation and inequality and eliminating $u(y)$, we obtain $$T^{+n}[u](x)+\sum_{i=0}^{n-1}S (x_i, x_{i+1})\leq v(y_{-n}) + \sum_{i={-n}}^{-1}S (y_{i}, y_{i+1}) $$ for any $(y_{-n},\dots,y_{-1}) \in (\mathbb{R}^d)^n$, $y_0=x_n=y$, $x_0=x$.
Since $T^{+n}[u](x)=T^{+n}\circ T^{-n}[v](x)=v(x)$, we conclude that $x\in\Sigma_-^n[v](y)$.
 \end{proof}
 
  \begin{lemma}\label{lemma:Pi2OfGraphSigma}
	 If $u\in T^-(C^0(\mathbb{T}^d))$, then $\Sigma_+[u]$ is surjective. In other words, \begin{align*}
    \bigcup_{x\in\mathbb{R}^d}\Sigma_+[u](x)=\mathbb{R}^d.
  \end{align*}
\end{lemma}
\begin{proof}
	Suppose that $u=T^-[v]$ for some $v\in C^0(\mathbb{T}^d)$. By Remark \ref{remark:coercive} (a), for any $y\in\mathbb{R}^d$, there exists $x\in\mathbb{R}^d$ such that $x\in\Sigma_-[v](y)$. It follows from Lemma \ref{lemma:Sigma-andSigma+} that $y\in \Sigma_+[u](x)$.
	\end{proof}
 
 For any  $u\in C^0(\mathbb{T}^d)$ and $x\in\mathbb{R}^d$, we have $$T^+\circ T^-[u](x)=\sup_{y\in\mathbb{R}^d}\inf_{z\in\mathbb{R}^d}u(z)+S(z,y)-S(x,y)\leq \sup_{y\in\mathbb{R}^d}u(x)=u(x).$$
 Similarly, for any $u\in C^0(\mathbb{T}^d)$ and $n\in\mathbb{Z}_+$, we have \begin{equation}\label{eq:T+T-nondecreasing}
 	T^{+n}\circ T^{-n}[u]\leq \cdots\leq  T^{+}\circ T^-[u] \leq u\leq  T^-\circ T^+[u] \leq \cdots\leq T^{-n}\circ T^{+n}[u].
 \end{equation} 
    Let $u\in C^0(\mathbb{T}^d)$, we denote by $\mathcal{O}^n[u]$ the closed set $$\mathcal{O}^n[u]:=\{x\in\mathbb{R}^d\mid T^{+n}\circ T^{-n}[u](x)=u(x)\}.$$ It follows from the equation \eqref{eq:T+T-nondecreasing} that $\mathcal{O}^{n+1}[u]\subset \mathcal{O}^n[u]$ for all $n\in\mathbb{Z}_+$. It follows from Remark \ref{remark:coercive} (a) and Lemma \ref{lemma:Sigma-andSigma+} that 
 	 $$\Sigma_-^n[u](\mathbb{R}^d)=\mathcal{O}^n[u], \quad  \forall n\in\mathbb{Z}_+.$$
 	 So $\mathcal{O}^n[u]$ is nonempty as the image of the multi-valued map $\Sigma_-^n[u]$.  Hence $$\mathcal{O}^\infty[u]:=\bigcap_{n\in\mathbb{Z}_+}\mathcal{O}^n[u]$$
 	is a nonempty closed subset of $\mathbb{R}^d$.
 	
 	The following proposition reveals that  $\mathcal{O}^\infty[u]$  can, to some extent, be regarded as a closed invariant set (see Proposition \ref{prop:invariant}).
 	\begin{proposition}
 		For any $u\in C^0(\mathbb{T}^d)$, $m\in\mathbb{Z}_+$, we have $$\Sigma_-^m[u](\mathcal{O}^\infty[T^{-m}[u]])=\mathcal{O}^\infty[u].$$
 	\end{proposition}
 	\begin{proof}
 		By Remark \ref{remark:coercive} (b), $$\Sigma_-^m[u]\circ\Sigma_-^n[T^{-m}[u]](\mathbb{R}^d)=\Sigma_-^{m+n}[u](\mathbb{R}^d), \quad \forall n\in\mathbb{Z}_+.$$
 		This implies that $$\Sigma_-^m[u](\mathcal{O}^n[T^{-m}[u]])=\mathcal{O}^{m+n}[u],\quad \forall n\in\mathbb{Z}_+.$$
 		Taking the limit as $n\to+\infty$, we conclude that $$\Sigma_-^m[u](\mathcal{O}^\infty[T^{-m}[u]])\subset\mathcal{O}^{\infty}[u].$$
 		Conversely, for any $x\in\mathcal{O}^{\infty}[u]$, by definition, $x\in\mathcal{O}^{m+n}[u]$ for all $n\in\mathbb{Z}_+$. Using Remark \ref{remark:coercive} (b) again, we have $x\in\Sigma_-^m[u](\mathcal{O}^n[T^{-m}[u]]) $ for all $n\in\mathbb{Z}_+$. Therefore, there exists a sequence $y_n\in \mathcal{O}^n[T^{-m}[u]]$ such that $x\in\Sigma_-^m[u](y_n).$ By Lemma \ref{lemma:Sigma-andSigma+}, $y_n\in \Sigma_+^m[T^{-m}[u]](x)$, which is a compact subset (depending on $x$). Hence, there exists a convergent subsequence $y_{n_k}\rightarrow y\in\Sigma_+^m[T^{-m}[u]](x)$. Since $x\in\mathcal{O}^{\infty}[u]$, by Lemma \ref{lemma:Sigma-andSigma+} again, we have $x\in \Sigma_-^m[u](y)$. Furthermore, since for any $n\in\mathbb{Z}_+$, $\mathcal{O}^{n+1}[T^{-m}[u]]\subset\mathcal{O}^n[T^{-m}[u]]$ and $\mathcal{O}^n[T^{-m}[u]]$ is closed, we conclude $y\in\mathcal{O}^\infty[T^{-m}[u]]$. Thus, \[\Sigma_-^m[u](\mathcal{O}^\infty[T^{-m}[u]])\supset\mathcal{O}^{\infty}[u].\qedhere\]
 	\end{proof}
 	
The following theorem is well-known:
 \begin{theorem}[Discrete Weak K.A.M. Theorem]\label{theorem:DWKT} There exists a unique constant $\overline{S}$ such that the equation $v=T^-[v]-\overline{S}$ admits at least one solution $v\in C^0(\mathbb{T}^d)$ and the equation $u=T^+[u]+\overline{S}$ admits at least one solution $u\in C^0(\mathbb{T}^d)$.
\end{theorem}

Such functions $v:\mathbb{R}^d\rightarrow\mathbb{R}$ are called \emph{negative discrete weak K.A.M solutions} and such functions $u:\mathbb{R}^d\rightarrow\mathbb{R}$ are called \emph{positive discrete weak K.A.M solutions}. The constant $\overline{S}$ is called the \emph{effective interaction}.

\begin{definition}
  A function $u\in C^0(\mathbb{T}^d)$ is called a \emph{sub-action} with respect to $S$ if \begin{align*}
    \forall (x,y)\in\mathbb{R}^{d}\times\mathbb{R}^{d} , \ \quad u(y)-u(x)\leq  S(x,y)-\overline{S},
  \end{align*}
  or equivalently, if $u\leq T^-[u]-\overline{S}$, or equivalently, if $u\geq T^+[u]+\overline{S}$. 
\end{definition}

\begin{definition}
  Let $u$ be a sub-action. We say that  $(x,y)\in \mathbb{R}^{d}\times \mathbb{R}^{d}$ is \emph{$u$-calibrated} if $u(y)-u(x)=S(x,y)-\overline{S}$.  A sequence $\{x_n\}_{n\in\mathbb{Z}}$ is called \emph{$u$-calibrated} if $(x_n,x_{n+1})$ is $u$-calibrated for all $n\in\mathbb{Z}$.  The \emph{non-strict set} $\mathcal{NS}[u]$ of $u$ collects all $u$-calibrated pairs:\begin{align*}
    \mathcal{NS}[u]:=\{(x,y)\in\mathbb{R}^{2}\mid u(y)-u(x)=S(x,y)-\overline{S}\}.
  \end{align*}
\end{definition}

\begin{remark}\label{remark:convergence} For any sub-action $u$, the functions $T^\mp[u]$ are also  sub-actions \cite[Proposition 1.2.6]{zavidovique2023}.  Thus, we have \begin{align*}
    u\leq T^-[u]-\overline{S}\leq T^{-2}[u]-2\overline{S}\leq T^{-3}[u]-3\overline{S}\leq \cdots
  \end{align*}
  This implies that the sequence  $\{T^{-n}[u]-n\overline{S}\}_{n\in\mathbb{Z}_+}$ is non-decreasing. Moreover, we can show that the sequence $\{T^{-n}[u]-n\overline{S}\}_{n\in\mathbb{Z}_+}$ is uniformly bounded \cite[Proposition 1.2.3]{zavidovique2023}. Note that (a) of Remark \ref{remark:LaxOleinik} states that the sequence $\{T^{-n}[u]-n\overline{S}\}_{n\in\mathbb{Z}_+}$ is equicontinuous. Therefore,  $\{T^{-n}[u]-n\overline{S}\}_{n\in\mathbb{Z}_+}$ converges uniformly to a function $u_-\in C^0(\mathbb{T}^d)$, which is a negative discrete weak K.A.M. solution.
\end{remark}

\begin{lemma}\label{lemma:u+=u-+}
	Let $u$ be a sub-action. Set $$u_-:=\lim_{n\to+\infty}T^{-n}[u]-n\overline{S}, \quad u_+:=\lim_{n\to+\infty}T^{+n}[u]+n\overline{S}$$ and $$u_{-+}:=\lim_{n\to+\infty}T^{+n}[u_-]+n\overline{S},\quad u_{+-}:=\lim_{n\to+\infty}T^{-n}[u_+]-n\overline{S}$$ Then $u_+=u_{-+}$ and $u_-=u_{+-}$.
\end{lemma}

\begin{proof}
We show the first equality.
	By Remark \ref{remark:convergence}, we have $u\leq u_-$ and thus $u_+\leq u_{-+}$. Letting $\varepsilon(n):=\sup (u_--T^{-n}[u]+n\overline{S})$ for any $n\in\mathbb{Z}_+$, then $$u_--\varepsilon(n)\leq T^{-n}[u]-n\overline{S}.$$
	Then using the definition of $T^{+}$, we have $$T^{+n}[u_-]-\varepsilon(n)\leq T^{+n}\circ T^{-n}[u]-n\overline{S}\leq u-n\overline{S}.$$
	Then for any $m\in\mathbb{Z}_+$, we have $$T^{+(n+m)}[u_-]+(n+m)\overline{S}-\varepsilon(n)\leq T^{+m}[u]+m\overline{S}.$$
	Letting $m\to+\infty$, we have $u_{-+}-\varepsilon(n)\leq u_+$ for any $n\in\mathbb{Z}_+$. Letting $n\to+\infty$, we have $u_{-+}\leq u_+$.
\end{proof}

 \begin{proposition}\label{proposition:ConjugatePair}
 	Let $u\in C^0(\mathbb{T}^d)$. Then $\{T^{+n}\circ T^{-n}[u]\}_{n\in\mathbb{Z}_+}$
 	converges uniformly to a function $\widetilde{u}_+\in C^0(\mathbb{T}^d)$ and  $\{T^{-n}\circ T^{+n}[u]\}_{n\in\mathbb{Z}_+}$
 	converges uniformly to a function $\widetilde{u}_-\in C^0(\mathbb{T}^d)$. Furthermore, if $u$ is a sub-action, then $$\widetilde{u}_+=u_+:=\lim_{n\to+\infty}T^{+n}[u]+n\overline{S}$$ is a positive discrete weak K.A.M. solution and $$\widetilde{u}_-=u_-:=\lim_{n\to+\infty}T^{-n}[u]-n\bar{S}$$ is a negative discrete weak K.A.M. solution. We call $(u_-,u_+)$ a conjugate pair. 
\end{proposition}

 \begin{proof}[Proof of Proposition \ref{proposition:ConjugatePair}]
 	Let us provide the proof for  the positive discrete weak K.A.M. solution and the the rest  can be obtain similarly. 
 	
 	Let $u\in C^0(\mathbb{T}^d)$.    It follows from the equation \eqref{eq:T+T-nondecreasing} that $\{T^{+n}\circ T^{-n}[u]\}_{n\in\mathbb{Z}_+}$ is non-increasing. It follows from item (a) of Remark \ref{remark:LaxOleinik} that $\{T^{+n}\circ T^{-n}[u]\}_{n\in\mathbb{Z}_+}$ is equicontinuous. It follows from  item (a) of Remark \ref{remark:LaxOleinik} that $\{T^{+n}\circ T^{-n}[u]\}_{n\in\mathbb{Z}_+}$ has uniformly bounded oscillation. Combining this with the fact that $\mathcal{O}^\infty[u]\ne\emptyset$, we obtain $\{T^{+n}\circ T^{-n}[u]\}_{n\in\mathbb{Z}_+}$ is uniformly bounded. Therefore,  $\{T^{+n}\circ T^{-n}[u]\}_{n\in\mathbb{Z}_+}$ converges uniformly to a function $\widetilde{u}_+\in C^0(\mathbb{T}^d)$.
 	
 	 Let $u$ be a sub-action. We show that $\widetilde{u}_+=u_{-+}:=\lim_{n\to+\infty}T^{+n}[u_-]+n\overline{S}$. Then by Lemma \ref{lemma:u+=u-+} we get $\widetilde{u}_+=u_+$. Letting $\varepsilon(n):=\sup (u_--T^{-n}[u]+n\overline{S})$ for any $n\in\mathbb{Z}_+$, then $$u_--\varepsilon(n)\leq T^{-n}[u]-n\overline{S}\leq u_-+\varepsilon(n).$$
 	 Applying the operator $T^{+n}$ to both sides of the inequality, we have $$T^{+n}[u_-]-\varepsilon(n)\leq T^{+n}\circ T^{-n}[u]-n\overline{S}\leq T^{+n}[u_-]+\varepsilon(n).$$
 	 Then $$T^{+n}[u_-]+n\overline{S}-\varepsilon(n)\leq T^{+n}\circ T^{-n}[u]\leq T^{+n}[u_-]+n\overline{S}+\varepsilon(n).$$
 	 Letting $n\to+\infty$, we obtain $\widetilde{u}_+=u_{-+}$.
 	  \end{proof}

  \begin{lemma}
  	Let $u$ be a sub-action. Then $$\{x\mid u_-(x)=u_+(x)\}=\mathcal{O}^\infty[u_-]\subset O^\infty[u].$$
  	  \end{lemma}
  
  \begin{proof}
  Since $\{T^{+n}\circ T^{-n}[u_-]\}_{n\in\mathbb{Z}_+}$ is non-increasing and converges to $u_+$, we obtain $\{x\mid u_-(x)=u_+(x)\}=\mathcal{O}^\infty[u_-]$ by the definition of $\mathcal{O}^\infty[u_-]$. Since $u_-\geq u$ and $\{T^{+n}\circ T^{-n}[u]\}_{n\in\mathbb{Z}_+}$ is also non-increasing and also converges to $u_+$, we obtain $\mathcal{O}^\infty[u_-]\subset O^\infty[u]$.
  	  \end{proof}
  	   We can directly deduce from the above lemma:
  \begin{proposition}
  	\[\bigcap_{(u_-,u_+) \text{ conjugate pair}}\{x\mid u_-(x)=u_+(x)\}=\bigcap_{u \text{ weak K.A.M.}}\mathcal{O}^\infty[u]=\bigcap_{u \text{ sub-action}}\mathcal{O}^\infty[u]. \]
  \end{proposition}
  
  Moreover, we have
  \begin{proposition}
  	$$\bigcap_{u \text{ sub-action}}\mathcal{O}^\infty[u]=\bigcap_{u \text{ sub-action}}\pi_1(\mathcal{NS}[u]).$$
  \end{proposition}
  
  \begin{proof}
  	Let $u$ be a sub-action. Then we have 
  	$$u(x)\geq T^+[u](x)+\overline{S}\geq u_+(x), \quad \forall x\in\mathbb{R}^d.$$
  	If $x\in \mathcal{O}^\infty[u]$, by Proposition \ref{proposition:ConjugatePair} and the definition of $\mathcal{O}^\infty[u]$, it follows that $u_+(x)=\widetilde{u}_+(x)=u(x)$.
  	  	Thus $$u(x)=T^+[u](x)+\overline{S},$$
  	which implies that $x\in\pi_1(NS[u])$.
  	
  	Let $u$ be a sub-action. If $x\in\pi_1(\mathcal{NS}[u])$ for any sub-action $u$, then  $x\in\pi_1(\mathcal{NS}[T^{+n}[u]])$ for any $n\in\mathbb{Z}_+$. Consequently, $$u(x)=T^{+n}[u](x)+n\overline{S}$$
  	for any $n\in\mathbb{Z}_+$. This implies $u(x)=u_+(x)$. Therefore, by Proposition \ref{proposition:ConjugatePair},  $u(x)=\widetilde{u}_+(x)$ for any $n\in\mathbb{Z}_+$. That is, $x\in \mathcal{O}^\infty[u]$. 
  \end{proof}

  The following proposition gives equivalent definitions of the projected Aubry set. 
 \begin{proposition}\label{prop:defofAubry}
 	Denote by $\mathcal{A}$ the projected Aubry set. Then \begin{align*}
 		\mathcal{A}&=\bigcap_{(u_-,u_+) \text{ conjugate pair}}\{x\mid u_-(x)=u_+(x)\}\\ 
 		&=\bigcap_{u \text{ weak K.A.M.}}\mathcal{O}^\infty[u]\\ 
 		&=\bigcap_{u \text{ sub-action}}\mathcal{O}^\infty[u]\\ 
 		&=\bigcap_{u \text{ sub-action}}\pi_1(\mathcal{NS}[u])\\ 
 		&=\pi_1\left(\bigcap_{u \text{ sub-action}}\mathcal{NS}[u]\right).
 	\end{align*}
 \end{proposition}
 
\begin{proof}
	It follows from \cite[Lemma 10.3]{GT11} and \cite[Proposition 10.5]{GT11} that \[\bigcap_{u \text{ sub-action}}\pi_1(\mathcal{NS}[u])=\pi_1\left(\bigcap_{u \text{ sub-action}}\mathcal{NS}[u]\right)=\mathcal{A}.  \qedhere\]
\end{proof}
 	
 We list below, without proof, some fundamental properties of the projected Aubry set $\mathcal{A}$ (see for instance \cite{BerJAMS,GT11, zavidovique2023}). 
\begin{proposition}\label{prop:PropertyOfAubry}  Assume that $S$ satisfies Hypothesis \ref{hypothesis:S1} and  Hypothesis \ref{hypothesis:S2}. 
  \begin{enumerate}
    \item [(i)] $\mathcal{A}$ is  closed, non-empty and $$\mathcal{A}=\pi_2\left(\bigcap_{u \text{ is a sub-action}}\mathcal{NS}(u)\right);$$
    \item  [(ii)] $\mathcal{A}+m=\mathcal{A},\ \forall m\in\mathbb{Z}^d$;
    \item [(iii)] Any sub-action $u$ is differentiable on $\mathcal{A}$. Moreover, the differential  $\d u:\mathcal{A}\rightarrow(\mathbb{R}^d)^*$ is independent of $u$;
    \item [(iv)] For any $x_0\in\mathcal{A}$, there exists a unique sequence $\{x_n\}_{n\in\mathbb{Z}}$ which is calibrated by any sub-action $u$;
    \item [(v)] Denote the \emph{Aubry set} $\mathcal{A}^*=\{(x,p)\mid p=\d u(x),x\in\mathcal{A}\}$. Then the projection $\pi_1:\mathcal{A}^*\rightarrow\mathcal{A} $ is a bi-Lipschitz homeomorphism.
  \end{enumerate}
\end{proposition}

\subsection{Regularity results in discrete weak K.A.M. theory}\label{subsection:Regularity}
In Section \ref{subsection:Regularity}, we operate under the assumption that the function  $S$  satisfies both Hypothesis \ref{hypothesis:S1} and Hypothesis \ref{hypothesis:S2}.

The image $T^-(B(\mathbb{T}^d))$ (resp. $T^+(B(\mathbb{T}^d))$) constitutes a family  of semiconcave (resp. semiconvex)  functions  with linear modulus and the same constant, as demonstrated in \cite[Proposition 2.4.10]{zavidovique2023}. 
In the next,  we will characterize the regularity of $u, T^-[u], T^+[u]$ in the following Lemmas~\ref{lemma:RegularityOfSubaction} and \ref{lemma:WellDefined}  whose proofs are given in the appendix.

\begin{lemma}\label{lemma:RegularityOfSubaction}
  Assume that $S$ satisfies Hypothesis \ref{hypothesis:S1} and Hypothesis \ref{hypothesis:S2}.  Let $u\in C^0(\mathbb{T}^d)$. \begin{enumerate}
      \item [(a)] If $x\in \Sigma_-[u](y)$ for some $(x,y)\in\mathbb{R}^{d}\times\mathbb{R}^{d}$, then \begin{align*}
          -\partial_1S(x,y)&\in \nabla^-u(x);\\
          \partial_2S(x,y)&\in\nabla^+T^-[u](y);
      \end{align*}
      \item [(b)] if $y\in \Sigma_+[u](x)$ for some $(x,y)\in\mathbb{R}^{d}\times\mathbb{R}^{d}$, then \begin{align*}
          -\partial_1S(x,y)&\in \nabla^-T^+[u](x);\\ 
          \partial_2S(x,y)&\in\nabla^+u(y); 
      \end{align*}
  \end{enumerate}
\end{lemma}

\begin{lemma}\label{lemma:injective}
	If $v\in SC(\mathbb{T}^d)$, then $\Sigma_-[v]$ is injective. In other words, if $\Sigma_-[v](y_1)\cap \Sigma_-[v](y_2)\ne\emptyset$, then $y_1=y_2$.
\end{lemma}
\begin{proof}
	Suppose that $x\in \Sigma_-[v](y_1)\cap \Sigma_-[v](y_2)$. By item (a) of Lemma \ref{lemma:RegularityOfSubaction}, we have $$\{-\partial_1S(x,y_1),-\partial_1S(x,y_1)\}\subset\nabla^-v(x).$$
	Since $v$ is semiconcave, $\nabla^+ v(x)$ is nonempty. Then, by item (iii) of Proposition \ref{prop:PropertyOfSemiconcave}, $v$ is differentiable at $x$, and $$\d v(x)=-\partial_1S(x,y_1)=-\partial_1S(x,y_2).$$
	Since the map $y\mapsto \partial_1S(x,y)$ is a homeomorphism (see item (2) of Hypothesis \ref{hypothesis:S2}), it follows that $y_1=y_2$.
\end{proof}

\begin{lemma}\label{lemma:visC11}
 If $v\in SC(\mathbb{T}^d)$, then $v$ is $C^{1,1}$ on $\mathcal{O}[v]$. In other words, $\d v(x)$ exists for all $x\in \mathcal{O}[v]$ and there exists $K>0$ such that \begin{align*}
    \|\d v(x_1)-\d v(x_2)\|\leq K\|x_1-x_2\|, \quad \forall x_1,x_2\in \mathcal{O}[v].
  \end{align*}
  Moreover, $T^+\circ T^-[v]$ takes the same values and derivatives as $v$ (so is also $C^{1,1}$) on $\mathcal{O}[v]$.
\end{lemma}

\begin{proof}
	For any $x\in \mathcal{O}[v]$,  we have \begin{equation}\label{eq:T+T-v=v}
		T^+\circ T^-[v](x)=v(x).
	\end{equation} Since $T^+\circ T^-[v]\leq v$, item(ii) of Proposition \ref{prop:PropertyOfSemiconcave} implies that $$\nabla^+T^+\circ T^-[v](x)\supset \nabla^+v(x)\quad \text{and}\quad \nabla^-T^+\circ T^-[v](x)\subset \nabla^-v(x).$$
	Because $T^+\circ T^-[v]$ is semiconvex and $v$ is semiconcave, all four superdifferentials and subdifferentials are nonempty. Then by item (iii) of Proposition \ref{prop:PropertyOfSemiconcave}, both $T^+\circ T^-[v]$ and $v$ are differentiable at $x$ and \begin{equation}\label{eq:dT+T-v=dv}
		\d T^+\circ T^-[v](x)=\d v (x).
	\end{equation}
	Since $v$ is a semiconcave function, item (i) of Proposition \ref{prop:PropertyOfSemiconcave} guarantees the existence of a constant $C\geq 0$ such that \begin{align*}
    v(x_1)-v(x_2)- \d v(x_2)(x_1-x_2)\leq \frac{C}{2}|x_1-x_2|^2, \quad \forall x_1, x_2 \in \mathcal{O}[v].
\end{align*} On the other hand, because $T^+\circ T^-[v]$ is semiconvex, there exists $C'\geq 0$ such that \begin{align*}
  T^+\circ T^-[v](x_1)-T^+\circ T^-[v](x_2)-\d T^+\circ  T^-[v]&(x_2)(x_1-x_2)\\ 
  & \geq -\frac{C'}{2}\|x_1-x_2\|^2, \quad \forall x_1, x_2 \in \mathcal{O}[v].
\end{align*}
Substituting equations \eqref{eq:T+T-v=v}  and \eqref{eq:dT+T-v=dv}, we  replace $T^+\circ T^-[v]$ and $\d T^+\circ T^-[v]$ with $v$ and  $\d v$ respectively. This gives \begin{align*}
  v(x_1)-v(x_2)-\d v(x_2)(x_1-x_2)\geq -\frac{C'}{2}\|x_1-x_2\|^2, \quad \forall x_1, x_2 \in \mathcal{O}[v].
\end{align*} Combining these inequalities, we obtain \begin{align*}
  -\frac{C'}{2}\|x_1-x_2\|^2\leq v(x_1)-v(x_2)- \d v(x_2)(x_1-x_2)\leq \frac{C}{2}\|x_1-x_2\|^2, \quad \forall x_1,x_2\in\mathcal{O}[v].
\end{align*}
Exchanging the roles of $x_1$ and $x_2$, we get \begin{align*}
  -\frac{C'}{2}\|x_1-x_2\|^2\leq v(x_2)-v(x_1)- \d v(x_1)(x_2-x_1)\leq \frac{C}{2}\|x_1-x_2\|^2, \quad \forall x_1,x_2\in\mathcal{O}[v].
\end{align*}
Adding the two inequalities yields \begin{align*}
  -C'\|x_1-x_2\|^2\leq (\d v (x_1)-\d v(x_2))(x_1-x_2)\leq C\|x_1-x_2\|^2, \quad \forall x_1,x_2\in\mathcal{O}[v].
\end{align*}
If we let $K=\max\{C,C'\}$, then 
\[
  \|\d v(x_1)-\d v(x_2)\|\leq K\|x_1-x_2\|, \quad \forall x_1,x_2\in\mathcal{O}[v].    \qedhere
\]
\end{proof}

\begin{lemma}\label{lemma:WellDefined}
  Assume that $S$ satisfies Hypothesis \ref{hypothesis:S1} and Hypothesis \ref{hypothesis:S2}.  Let $u\in C^0(\mathbb{T}^d)$.
  \begin{enumerate}
    \item [(a)] Let $y\in\mathbb{R}^d$. Then  there exists a unique $x\in\mathbb{R}^d$ such that  $T^-[u](y)=u(x)+S(x,y)$ if and only if $T^-[u]$ is differentiable at $y$. 
    \item [(b)] Let $x\in\mathbb{R}^d$. Then there exists a unique $y\in\mathbb{R}^d$ such that  $T^+[u](x)=u(y)-S(x,y)$  if and only if $T^+[u]$ is differentiable at $x$.
  \end{enumerate} 
\end{lemma}

Note that a version of Lemma~\ref{lemma:WellDefined} for discrete weak K.A.M. solutions could be referred in \cite{Arnaud16}.

\begin{lemma}\label{lemma:Sigma+Lip}
	If $u=T^-[v]$, $v\in SC(\mathbb{T}^d)$ , then $\Sigma_+[u]$ is a continuous mapping on $ \mathcal{O}[v]$. Moreover, if the discrete standard map $F$ is locally Lipschitz  continuous, then   $\Sigma_+[u]$ is Lipschitz  continuous  on $ \mathcal{O}[v]$.
\end{lemma}

\begin{proof}
	For any $x\in \mathcal{O}[v]$, since $v$ is semiconcave, it follows from Lemma \ref{lemma:visC11} that $T^+\circ T^-[v]=T^+[u]$ is differentiable at $x$. Then, by item (b) of Lemma \ref{lemma:WellDefined}, the set $\Sigma_+[u](x)$ is a singleton. Consequently, $\Sigma_+[u]$ can be viewed as a real-valued mapping on $\mathcal{O}[v]$. 
	
	For any $x\in \mathcal{O}[v]$, if $y=\Sigma_+[u](x)$, then, by (b) of Lemma \ref{lemma:RegularityOfSubaction} and the differentiability of $T^+[u]$ at $x$, we have $-\partial_1S(x,y)=\d T^+[u](x)$. From Lemma \ref{lemma:visC11}, which states $\d T^+[v](x)=\d v(x)$, it follows that $-\partial_1S(x,y)=\d v(x)$. Using the definition of the discrete standard map $F$ (see  Remark \ref{remark:F}), we deduce that $$y=\pi_1\circ F(x,\d v(x)).$$
	Thus, the continuity of $\Sigma_+[u]$ follows from the continuity of $F$ and $\d v$. 
	
	Now assume that $F$ is locally Lipschitz  continuous, we will show that $\Sigma_+[u]$ is Lipschitz  continuous  on $\mathcal{O}[v]$. By item (v) of Proposition \ref{prop:PropertyOfSemiconcave}, there exists $M>0$ such that $\|\d v (x)\|\leq M$ for any $x\in \mathcal{O}[v]$. Let $B_M:=\{p\in(\mathbb{R}^d)^*\mid \|p\|\leq M\}$. Since  $F$ is locally Lipschitz, the mapping $\pi_1\circ F$ is Lipschitz continuous on the compact set $\mathbb{T}^d\times B_M$. Denote the Lipschitz constant of $\pi_1\circ F$ by $K'$. Then, for any  $x_1,x_2\in \mathcal{O}[v]$, we have \begin{align*}
    \|\Sigma_+[u](x_1)-\Sigma_+[u](x_2)\|&=\left\|\pi_1\circ F(x_1, \d v(x_1))-\pi_1\circ F(x_2, \d v(x_2))\right\|\\ 
    &\leq  K'\|(x_1, \d u(x_1))-(x_2, \d u(x_2))\|_{\mathbb{R}^{2d}}.  \end{align*}
  Finally, by Lemma \ref{lemma:visC11}, we have
\[
\|(x_1, \d v(x_1)) - (x_2, \d v(x_2))\|_{\mathbb{R}^{2d}} \leq \sqrt{1 + K^2} \|x_1 - x_2\|,
\]
where  $K$  is the Lipschitz constant of \( \d v \). Combining these, we conclude that $$\|\Sigma_+[u](x_1)-\Sigma_+[u](x_2)\|\leq K'\sqrt{1 + K^2} \|x_1 - x_2\|.$$
Thus, $\Sigma_+[u]$ is Lipschitz continuous on $ \mathcal{O}[v]$.
\end{proof}

\begin{proposition}\label{prop:invariant} Let $u$ be a discrete weak K.A.M. solution. Define $\mathcal{O}^*[u]:=\{(x,\d u(x))\mid x\in\mathcal{O}^\infty[u]\}$. Then, it is an $F$-invariant set, i.e.,
	$$F(\mathcal{O}^{*}[u])=\mathcal{O}^{*}[u].$$
\end{proposition}

\begin{proof}
Since $u$ is a discrete weak K.A.M. solution, $T^{-}[u]=u+\bar{S}$.

	{\textbf{Step 1. Show $F(\mathcal{O}^{*}[u])\subset\mathcal{O}^{*}[u]$:}} For any $x\in\mathcal{O}^\infty[u]$, by Proposition 2.8, $x\in\Sigma_-[u](\mathcal{O}^\infty[u]).$
	That is, there exists $y\in \mathcal{O}^\infty[u]$ such that $x\in\Sigma_-[u](y)$. By Lemma \ref{lemma:Sigma-andSigma+}, $y\in \Sigma_+[u](x)$. Using Lemma \ref{lemma:RegularityOfSubaction}, we have $$-\partial_1S(x,y)=\d u (x)\quad \text{and}\quad \partial_2S(x,y)=\d u(y).$$ Hence, $F(\mathcal{O}^{*}[u])\subset\mathcal{O}^{*}[u].$
	
	{\textbf{Step 2. Show $\mathcal{O}^{*}[u]\subset F(\mathcal{O}^{*}[u])$:}} For any $x\in\mathcal{O}^\infty[u]$, let $w\in \Sigma_-[u](x)$. By Proposition 2.8, $w\in\mathcal{O}^\infty[u]$. By Lemma \ref{lemma:Sigma-andSigma+}, $x\in \Sigma_+[u](w)$. Using Lemma \ref{lemma:RegularityOfSubaction}, we have  $$-\partial_1S(w,x)=\d u (w)\quad \text{and}\quad \partial_2S(w,x)=\d u(x).$$ Hence, $F(\mathcal{O}^{*}[u])\supset\mathcal{O}^{*}[u].$
	
	Combining both inclusions, we conclude that
$
F(\mathcal{O}^{*}[u])=\mathcal{O}^{*}[u].$
\end{proof}

\section{Proofs of Theorems \ref{thm:main}-\ref{thm:packaged}}\label{sec:Theorem1}
 In Section \ref{sec:Theorem1} , we operate under the assumption that the function  $S$  satisfies Hypothesis \ref{hypothesis:S1},  Hypothesis \ref{hypothesis:S2} and  Hypothesis \ref{hypothesis:S3}.
The proof of Theorem~\ref{thm:main} is divided into the following four subsections and then Theorem~\ref{thm:packaged} follows immediately.  

\subsection{$\Sigma_+[u]$ is a non-decreasing continuous real-valued mapping}

Let $u\in C^0(\mathbb{T})$. Recall that the \emph{optimal forward mapping} is the closed multi-valued mapping \begin{align*}
    \Sigma_+[u]:\mathbb{R}&\rightarrow\mathbb{R}\\ 
    x&\mapsto \mathop{\arg\max}_{y\in\mathbb{R}}\{u(y)-S(x,y)\}.
  \end{align*}
  So, without verification, $\Sigma_+[u](x)$ is a subset of $\mathbb{R}$. However, we will later show that under certain assumptions, $\Sigma_+[u](x)$ is actually a singleton for any $x\in\mathbb{R}$ (see Lemma \ref{lemma:singleton}). In other words, the optimal forward mapping is actually a mapping under certain assumptions.

Thanks to Remark \ref{remark:coercive} (a), we have $\Sigma_+[u](x)$ is non-empty for each $x\in\mathbb{R}$.  For two subsets $A\subset \mathbb{R} $ and $B\subset \mathbb{R}$, we define $
A\leq B$ if and only if $ x_1\leq x_2, \forall x_1\in A, x_2\in B.
$ 
\begin{lemma}\label{lemma:SigmaNondecreasing}
  If $u\in C^0(\mathbb{T})$, then $\Sigma_+[u]$ is non-decreasing. In other words, if $x_1<x_2$, then $\Sigma_+[u](x_1)\leq \Sigma_+[u](x_2)$.
\end{lemma}
\begin{proof}
Suppose that $x_1<x_2$ but $y_1>y_2$, where $y_1\in\Sigma_+(x_1)$ and $y_2\in \Sigma_+(x_2)$. We will use the Non-crossing Lemma (see Hypothesis \ref{hypothesis:S3}) to reduce a contradiction. We know $(x_1-x_2)(y_1-y_2)<0$ and the Non-crossing Lemma says \begin{align*}
  S(x_1,y_2)+S(x_2,y_1)<S(x_1,y_1)+S(x_2,y_2).
\end{align*} Using the definition of the optimal forward mapping $\Sigma_+[u]$, from $y_1\in\Sigma_+(x_1)$ and $y_2\in \Sigma_+(x_2)$, we get that \begin{align*}
    u(y_1)-S(x_1,y_1)\geq u(y_2)-S(x_1,y_2);\\ 
    u(y_2)-S(x_2,y_2)\geq u(y_1)-S(x_2,y_1).
\end{align*}
Adding these two inequalities together, we get \begin{align*}
    S(x_1,y_2)+S(x_2,y_1)\geq S(x_1,y_1)+S(x_2,y_2).
\end{align*}
This contradicts the Non-crossing Lemma and thus we complete the proof.
\end{proof}

	\begin{lemma}\label{lemma:singleton}
   If $u= T^-[v]$, $v\in SC(\mathbb{T})$, then for any $x\in\mathbb{R}$, the set $\Sigma_+[u](x)$ is a singleton.
\end{lemma}

\begin{proof}
  Suppose that there exists $y_1<y_2$ such that $y_1,y_2\in\Sigma_+[u](x_0)$.  By Lemma \ref{lemma:SigmaNondecreasing},  for any $x>x_0$, we have $\Sigma_+[u](x)\geq y_2$ and for any $x<x_0$, we have $\Sigma_+[u](x)\leq y_1$. In other words, if $\Sigma_+[u](x)\cap (y_1,y_2)\ne\emptyset$, then $x=x_0$. 
 	Now, take $z_1,z_2$ such that $y_1<z_1<z_2<y_2$. For any $x\in \Sigma_-[v](z_i)$, $i=1,2$, by Lemma \ref{lemma:Sigma-andSigma+}, we have $z_i\in \Sigma_+[u](x)\cap (y_1,y_2)$, which implies $x=x_0$. 
  	Thus, $x_0\in \Sigma_-[v](z_1)\cap \Sigma_-[v](z_2)$. By Lemma \ref{lemma:injective}, since $v$ is semiconcave, we must have $z_1=z_2$. This leads to a contradiction. \end{proof}

If $u= T^-[v]$, $v\in SC(\mathbb{T})$, then $\Sigma_+[u]:\mathbb{R}\rightarrow\mathbb{R}$ is a non-decreasing real-valued mapping. Since discontinuities of monotonic function are jump discontinuities, but Lemma \ref{lemma:Pi2OfGraphSigma} implies that $\Sigma_+[u]$ has no jump discontinuity, we conclude that $\Sigma_+[u]$ is continuous for any $u\in T^-(C^0(\mathbb{T}))$.
Utilizing the periodicity of $u$, along with  the diagonal-periodicity of $S$, we deduce from the definition that  $\Sigma_+[u](x+1)=\Sigma_+[u](x)+1$. In other words, $\Sigma_+[u]$ is a lift of a circle map (not necessarily invertible) of degree $1$.
We conclude that 

\begin{proposition}\label{prop:LiftOfCircleMap}
	If $u= T^-[v]$, $v\in SC(\mathbb{T})$, then $\Sigma_+[u]$ is a non-decreasing continuous real-valued mapping, which is a lift of a circle map of degree $1$.\qed
\end{proposition}

\subsection{$\Sigma_+[u]$ is  Lipschitz continuous}

 
 Let  $u=T^-[v]$, where $v\in SC(\mathbb{T})$. It has already been established that $\Sigma_+[u]$ is  Lipschitz continuous on the subset $\mathcal{O}[v]=\Sigma_-[v](\mathbb{R})\subset \mathbb{R}$ (see Lemma \ref{lemma:Sigma+Lip}). To prove $\Sigma_+[u]$ is  Lipschitz  continuous on the entire $\mathbb{R}$, we proceed as follows:
 
  For any $x\in\mathbb{R}$, we will identify another point  $x'\in \Sigma_-[v](\mathbb{R})$ such that $\Sigma_+[u](x)=\Sigma_+[u](x')$. Then, by leveraging the established Lipschitz property of $\Sigma_+[u]$ on $\Sigma_-[v](\mathbb{R})$,  we will deduce the Lipschitz continuity of $\Sigma_+[u]$ at $x$.

\begin{lemma}\label{lemma:SingularityAndNSu}
	Let $u= T^-[v]$, $v\in SC(\mathbb{T})$. For any $y\in\mathbb{R}$, the preimage $(\Sigma_+[u])^{-1}(y)$ is either a singleton or a non-degenerate closed bounded interval. If $(\Sigma_+[u])^{-1}(y)$ is a singleton, denoted by $x$, then $x\in\Sigma_-[v](y)$. If $(\Sigma_+[u])^{-1}(y)$ is a non-degenerate closed bounded interval, denoted by $[x_-,x_+]$, then $x_-,x_+\in\Sigma_-[v](y)$.
\end{lemma}

\begin{proof}
	By Lemma \ref{lemma:Pi2OfGraphSigma}, we know that for any $y\in\mathbb{R}$, the set $(\Sigma_+[u])^{-1}(y)$ is nonempty. 

  {\textbf{Case 1: }}If $(\Sigma_+[u])^{-1}(y)$ is a singleton $\{x\}$. By Remark \ref{remark:coercive} (a), we can find $x'\in\mathbb{R}$ such that $x'\in\Sigma_-[u](y)$. Then by Lemma \ref{lemma:Sigma-andSigma+}, $x'\in(\Sigma_+[u])^{-1}(y)$. Since $(\Sigma_+[u])^{-1}(y)$ is a singleton, we must have $x'=x$, implying $x\in\Sigma_-[u](y)$.

  {\textbf{Case 2: }}If $(\Sigma_+[u])^{-1}(y)$ contains at least two distinct points $x_1<x_2$.  Since $\Sigma_+[u]$ is continuous and satisfies $\Sigma_+[u](x+1)=\Sigma_+[u](x)+1$, we define\begin{align*}
    x_-:=\inf\{x\in\mathbb{R}\mid \Sigma_+[u](x)=y\},\quad \quad x_+:=\sup\{x\in\mathbb{R}\mid \Sigma_+[u](x)=y\}.
  \end{align*}
  Clearly, $-\infty<x_-\leq x_1<x_2\leq x_+<+\infty$ and $x_-,x_+\in (\Sigma_+[u])^{-1}(y)$. Since $\Sigma_+[u]$ is non-decreasing, $(\Sigma_+[u])^{-1}(y)=[x_-,x_+]$, which is a non-degenerate closed bounded interval.

 To prove $x_-,x_+\in\Sigma_-[v](y)$,  we first establish that the cardinality of \begin{align*}
    \mathcal{C}:=\{y\mid (\Sigma_+[u])^{-1}(y) \text{ is a non-degenerate closed interval}\}
  \end{align*} is at most countable. 
  Define a mapping \begin{align*}
    (\Sigma_+[u])^{-1}:\mathcal{C}&\rightarrow 2^{\mathbb{R}}\\
    y&\mapsto [x_-(y),x_+(y)].
  \end{align*}
 For $y_1\ne y_2$, the intervals $ [x_-(y_1),x_+(y_1)]$ and $[x_-(y_2),x_+(y_2)]$ are disjoint.
  Using the Axiom of Choice, we can select a rational number $q(y)$ from each $[x_-(y),x_+(y)]$, thus defining an injection from $\mathcal{C}$ into $\mathbb{Q}$.Since  $\mathbb{Q}$  is countable,  $\mathcal{C}$  is also at most countable.

	Now, we show $x_-,x_+\in\Sigma_-[v](y)$.
  Suppose $(\Sigma_+[u])^{-1}(y)=[x_-,x_+]$ is a non-degenerate interval. We claim there exists a sequence $\{x_n\}_{n\in\mathbb{Z}_+}$ with $x_n\rightarrow x_-$ and   $(\Sigma_+[u])^{-1}(\Sigma_+[u](x_n))$  is a singleton for each  $n$. Suppose this is not true, then there exists  $\delta>0$ such that $(\Sigma_+[u])^{-1}(\Sigma_+[u](x))$ is not a singleton for all $x\in [x_--\delta,x_-]$. Since  $x_-=\inf\{x\mid \Sigma_+[u](x)=y\}$, we have $\Sigma_+[u](x_--\delta)<y$. Using the countability of $\mathcal{C}$, there exists $y'$ such that $\Sigma_+(x_--\delta)<y'<y$ and $(\Sigma_+[u])^{-1}(y')$ is a singleton, say $x'$. Then since $\Sigma_+$ is non-decreasing, $x'\in [x_--\delta,x_-]$, which contradicts the assumption that $(\Sigma_+[u])^{-1}(y')$ is not a singleton.

  	Define $y_n:= \Sigma_+[u](x_n)$, and thus  $x_n\in\Sigma_-[v](y_n)$ for all $n$. Define $f_n(x):=v(x)+S(x,y_n)$ and $f(x):=v(x)+S(x,y)$. Since $f_n\rightarrow f$ uniformly on compact sets, $\inf f_n$ converges to $\inf f$. Note that $\inf f_n=v(x_n)+S(x_n,y_n)$, so $$\inf f=\lim_n\inf f_n=\lim_n\left(v(x_n)+S(x_n,y_n)\right)=v(x_-)+S(x_-,y).$$
  		Hence $x_-\in\Sigma_-[v](y)$. A similar argument shows $x_+\in\Sigma_-[v](y)$. 
\end{proof}

\begin{proposition}\label{proposition:SigmaLipschitz}
  If $u= T^-[v]$, $v\in SC(\mathbb{T})$ and the discrete standard map $F$ is locally Lipschitz  continuous, then $\Sigma_+[u]$ is Lipschitz  continuous.
\end{proposition}

\begin{proof}
  By Lemma \ref{lemma:Sigma+Lip}, we know that  $\Sigma_+[u]$ is Lipschitz  continuous  on $\Sigma_-[v](\mathbb{R})$. Let the Lipschitz constant of $\Sigma_+[u]$ on $\Sigma_-[v](\mathbb{R})$  be denoted by $K''$. For any $x_1\leq x_2\in\mathbb{R}$, by Lemma \ref{lemma:SingularityAndNSu}, there exist points $x_1',x_2'\in\Sigma_-[v](\mathbb{R})$ such that \begin{align*}
    x_1\leq x_1'\leq x_2'\leq x_2,\quad  \Sigma_+[u](x_1')=\Sigma_+[u](x_1)\quad\text{ and }\quad \Sigma_+[u](x_2')=\Sigma_+[u](x_2).
  \end{align*} 
  Then, $$|\Sigma_+[u](x_1)-\Sigma_+[u](x_2)|=|\Sigma_+[u](x_1')-\Sigma_+[u](x_2')|.$$
   Since $x_1',x_2'\in \Sigma_-[v](\mathbb{R})$ and $\Sigma_+[u]$ is Lipschitz continuous on $\Sigma_-[v](\mathbb{R})$ with constant $K''$, we have $$|\Sigma_+[u](x_1')-\Sigma_+[u](x_2')|\leq K''|x_1'-x_2'|.$$
   From the relationship $x_1\leq x_1'\leq x_2'\leq x_2$, it follows that $$|x_1'-x_2'|\leq |x_1-x_2|.$$ Therefore,
  \begin{align*}
    |\Sigma_+[u](x_1)-\Sigma_+[u](x_2)|\leq K''|x_1-x_2|.
  \end{align*}
  This proves that  $\Sigma_+[u]$ is  Lipschitz  continuous  on the whole $\mathbb{R}$.
\end{proof}

\subsection{$\Sigma_+[u]$ propagates singularities}

We first give the following equivalent condition for singular points of functions in $T^-(SC(\mathbb{T}))$.

\begin{proposition}\label{prop:SingularityAndGraph}
  Let $u=T^-[v]$, where $v\in SC(\mathbb{T})$.  Then $y\in\operatorname{Sing}(u)$ if and only if  $(\Sigma_+[u])^{-1}(y)$ is a non-degenerate closed bounded interval.
\end{proposition}

\begin{proof}
Suppose that $y\in\operatorname{Sing}(u)$. Since $u=T^-[v]$ is  non-differentiable at $y$, by part (a) of  Lemma \ref{lemma:WellDefined}, there exist distinct $x_1 \ne x_2$ such that $\{x_1,x_2\}\subset \Sigma_-[v](y)$. Using Lemma~\ref{lemma:Sigma-andSigma+}, it follows that $\{x_1,x_2\}\subset (\Sigma_+[u])^{-1}(y)$.  By Lemma \ref{lemma:SingularityAndNSu}, $(\Sigma_+[u])^{-1}(y)$ must be a non-degenerate closed bounded interval.

 Suppose that $(\Sigma_+[u])^{-1}(y)$ is a non-degenerate closed bounded interval $[x_-,x_+]$. By Lemma \ref{lemma:SingularityAndNSu}, $x_-,x_+\in \Sigma_-[v](y)$. Using part (a) of Lemma \ref{lemma:WellDefined}, $T^-[v]$ is non-differentiable at $y$. Hence, $y\in\operatorname{Sing}(u)$.
 \end{proof}

\begin{proposition}\label{prop:SingularityPropagate}
  Let $u=T^-[v]$, where $v\in SC(\mathbb{T})$. If $x\in\operatorname{Sing}(v)$, then $\Sigma_+[u](x)\in\operatorname{Sing}(u)$.
\end{proposition}

\begin{proof}
  	Let $x\in\operatorname{Sing}(v)$ and $y=\Sigma_+[u](x)$. Take any $x'\in \Sigma_-[v](y)$. By Lemma \ref{lemma:visC11}, $v$ is differentiable at $x'$, so $x'\ne x$. By Lemma  \ref{lemma:Sigma-andSigma+}, $\Sigma_+[u](x')=y$, implying $\{x,x'\}\subset (\Sigma_+[u])^{-1}(y)$. By Lemma \ref{lemma:SingularityAndNSu}, 	$(\Sigma_+[u])^{-1}(y)$ is a non-degenerate closed bounded interval. Hence by Proposition \ref{prop:SingularityAndGraph}, $y\in\operatorname{Sing}(u)$.
  	  \end{proof}
  	  
  \begin{corollary}
  	Let $u$ be a discrete weak K.A.M. solution. If $x\in\operatorname{Sing}(u)$, then $\Sigma_+[u](x)\in\operatorname{Sing}(u)$.
  \end{corollary}

\subsection{$\Sigma_+[u]$ admits a rotation number}
The following lemma is well-known and we provide its proof for selfcontainedness. 
\begin{lemma}\label{lemma:RotationNumber}
    Let $\Phi:\mathbb{R}\rightarrow\mathbb{R}$ be a continuous map such that \begin{enumerate}
        \item [(i)] $\Phi(x+1)=\Phi(x)+1$, for any $x\in\mathbb{R}$;
        \item [(ii)] if $x_1<x_2$, then $\Phi(x_1)\leq \Phi(x_2)$.
    \end{enumerate}
    Then the limit \begin{align*}
        \rho(\Phi):=\lim_{n\to+\infty}\frac{1}{n}\left(\Phi^{n}(x)-x\right)
    \end{align*}
    exists for all $x\in\mathbb{R}$ and is independent of $x$.
\end{lemma}

\begin{proof}
  From property (i), by induction, we have: \begin{align*}
     \Phi^m(x+1)=\Phi^m(x)+1,\quad  \forall x\in\mathbb{R}, m\in\mathbb{N}_+.
  \end{align*}
  Thus, $\Phi^m-\operatorname{id}$ is $1$-periodic. Define \begin{align*}
    \alpha_m:=\inf_{\mathbb{R}} \{\Phi^m-\operatorname{id}\},\quad \quad \beta_m:=\sup_{\mathbb{R}} \{\Phi^m-\operatorname{id}\}.
  \end{align*}
  From property (ii), for any $x_1\leq  x_2\leq x_1+1$,   we have:$$\Phi^m(x_1)\leq \Phi^m(x_2)\leq \Phi^m(x_1)+1.$$ This implies 
 $$|(\Phi^m(x_1)-x_1)-(\Phi^m(x_2)-x_2)|\leq 1,$$ and hence $\beta_m-\alpha_m\leq 1$.

  Let  $n=qm+r$, where $ 0\leq r<m$. Then, using $ \alpha_m\leq \Phi^m(x)-x\leq \beta_m,$ we obtain: \begin{align*}
    q\alpha_m\leq \Phi^{qm}(x)-x\leq q\beta_m.
  \end{align*}
  Additionally, $$r\alpha_1\leq \Phi^r(x)-x \leq r\beta_1.$$
  Combining these, we have:\begin{align*}
    q\alpha_m+r\alpha_1\leq \Phi^{n}(x)-x \leq q\beta_m+r\beta_1.
  \end{align*}
  Dividing through by $n$, we get:
   \begin{align*}
    \frac{q\alpha_m+r\alpha_1}{n}\leq \frac{\Phi^{n}(x)-x}{n}\leq \frac{q\beta_m+r\beta_1}{n}.
  \end{align*}
  As $n\to+\infty$, $q/n\to 1/m$, and the terms involving $r/n$ vanish. Moreover, since $\beta_m-\alpha_m\leq 1$, the difference between the bounds $(\beta_m-\alpha_m)/m$ vanishes as $m\to+\infty$.
  Thus $\Lim{n\to+\infty}\frac{1}{n}\left(\Phi^{n}(x)-x\right)$ exists and is independent of $x$.
\end{proof}

By combining Proposition \ref{prop:LiftOfCircleMap}  and Lemma \ref{lemma:RotationNumber}, we proved that if $u\in T^-(SC(\mathbb{T}))$, then $\Sigma_+[u]$ admits a rotation number $\rho(\Sigma_+[u])$ for positive iterations for any $u\in T^-[SC(\mathbb{T})]$.

\begin{proposition}\label{prop:RotationNumber}
	If $u\in T^-(SC(\mathbb{T}))$ and $u$ is a sub-action, then $\rho(\Sigma_+[u])$ is the rotation number of the Aubry set.
\end{proposition}

\begin{proof}
  	Let $\{x_n\}_{n\in\mathbb{Z}}$ be a $u$-calibrated sequence contained in the projected Aubry set. For any $n\in\mathbb{Z}$, we have: $$u(x_{n+1})-S(x_n,x_{n+1})=u(x_n)-\overline{S}.$$
  	 Since $u$ is a sub-action, for any $y\in\mathbb{R}$:  \begin{align*}
  		u(x_n)-\overline{S}
  		\geq u(y)-S(x_n,y).
  	\end{align*} 
  	Thus: $$u(x_{n+1})-S(x_n,x_{n+1})\geq u(y)-S(x_n,y),$$
  	which implies $x_{n+1}\in\Sigma_+[u](x_{n})$. By Lemma \ref{lemma:singleton},  $x_{n+1}=\Sigma_+[u](x_{n})$. The rotation number is given by:$$\lim_{n\to+\infty}\frac{1}{n}(\Sigma_+^n[u](x_0)-x_0)=\lim_{n\to+\infty}\frac{1}{n}(x_n-x_0).$$
  	This is the rotation number of the Aubry set.  	  
  \end{proof}
  
  \begin{corollary}
  	Let $u$ be a discrete weak K.A.M. solution. Then $\operatorname{Sing}(u)$ is  a forward invariant set under $\Sigma_+[u]$, with the same rotation number as the Aubry set.
  \end{corollary}

\section{Proofs of Theorems \ref{thm:regularization}-\ref{thm:main3}}\label{sec:consequences}

In this section, we study several consequences of the dynamics $\Sigma_+[u]$ and prove the associated theorems. We operate under the assumption that the function $S$ satisfies Hypothesis~\ref{hypothesis:S1}, Hypothesis \ref{hypothesis:S2}, and Hypothesis \ref{hypothesis:S3}. The structure of this section is as follows:\begin{itemize}
	\item In Section 4.1, we prove Theorem~\ref{thm:regularization}.
	\item In Section 4.2, we prove Theorem~\ref{thm:main2}.
	\item In Section 4.3, we prove Theorems~\ref{thm:SingularDynamics} and \ref{thm:main3}.
\end{itemize}
By analyzing the structure and properties related to the dynamics $\Sigma_+[u]$, we aim to uncover its significant behaviors and applications under different scenarios.

\subsection{Regularization results}

This result serves as a discrete analogue of Patrick Bernard’s regularization theorem \cite{Ber07}. Notice that \cite{zavidovique2023} states that if $u$ is a discrete weak K.A.M. solution, then $T^+[u]$ is  $C^1$.

\begin{proposition}\label{prop:Regularization}
	If $v\in SC(\mathbb{T})$, then $T^+\circ T^-[v]$ is $C^1$. Moreover, if the discrete standard map $F$ is  locally bi-Lipschitz, then $T^+\circ T^-[v]$ is $C^{1,1}$.	
	\end{proposition}

\begin{proof}
From Lemma \ref{lemma:WellDefined} (item (b)), for any $x\in\mathbb{R}^d$, $\Sigma_+[u](x)$ is a singleton if and only if $T^+[u]$ is differentiable at $x$. Let $u:=T^-[v]$.  By  Lemmas \ref{lemma:singleton}  and \ref{lemma:WellDefined}, $\Sigma_+[u](x)$ is a singleton for all $x\in\mathbb{R}^d$. This implies that $T^+[u]$ is differentiable everywhere. From item (b) of Lemma \ref{lemma:RegularityOfSubaction}, we have: \begin{equation}\label{eq:-partialS=dT+}
	-\partial_1S(x,\Sigma_+[u](x))=\d T^+[u](x), \quad \forall x\in\mathbb{R}.
\end{equation}
Since $\partial_1S$ and $\Sigma_+[u]$ are continuous, it follows that $T^+[u]$ is $C^1$. 

Assume $F$ is locally bi-Lipschitz, meaning both $F$ and $F^{-1}$ are locally Lipschitz  continuous. Recall that $$\partial_1S(x,y)=-\pi_2\circ F^{-1}(y,\pi_2\circ F(x,y)).$$ Since $F^{-1}$ is locally Lipschitz  continuous, $\partial_1S$ is also locally Lipschitz continuous.

	Denote $M:=\|\Sigma_+[u]\|_{C^0[0,1]}$. On the compact set  $[0,1]\times [-M,M]$, $\partial_1S$ is Lipschitz  continuous. From Lemma \ref{proposition:SigmaLipschitz}, $\Sigma_+[u]$ is Lipschitz  continuous  on $[0,1]$. Using the equation \eqref{eq:-partialS=dT+}, we conclude that $\d T^+[u]$ is Lipschitz  continuous  on $[0,1]$. Since $T^+[u]$ is $1$-periodic, the Lipschitz continuity of  $\d T^+[u]$ extends to $\mathbb{R}$. Thus, $T^+[u]$ is $C^{1,1}$ globally.
\end{proof}

\begin{corollary}
If the discrete standard map $F$ is locally bi-Lipschitz, then for any bounded function $u:\mathbb{T}\rightarrow\mathbb{R}$, $T^{+}\circ T^{-2}[u]$ is $C^{1,1}$.
\end{corollary}


\subsection{An explanation of Arnaud's observation}

To prove Theorem \ref{thm:main2}, we first derive an explicit formula for $(\Sigma_+[u])^{-1}$ as stated in Lemma \ref{lemma:EquivalentDefinitionOfSingularities}.

\begin{lemma}\label{lemma:EquivalentDefinitionOfSingularities}
  Let $u=T^-[v]$ with $v\in SC(\mathbb{T})$ and assume $F$ is the discrete standard map. Then for any $y\in\mathbb{R}$,  \begin{align*}
    (\Sigma_+[u])^{-1}(y)=\pi_1\circ F^{-1}(y,\nabla^+ u(y)).
  \end{align*}
\end{lemma}
\begin{proof}
We consider two cases based on the differentiability of $u$ at $y$.

  {\textbf{Case 1: }}$u$ is differentiable at $y$. By Proposition  \ref{prop:SingularityAndGraph}, $(\Sigma_+[u])^{-1}(y)=\{x\}$, a singleton. From Lemma \ref{lemma:RegularityOfSubaction} (b), $\partial_2S(x,y)=\d u(y)$. Using the definition of the discrete standard map (see Remark \ref{remark:F}), we conclude: $$x=\pi_1\circ F^{-1}(y,\d u(y)).$$

  {\textbf{Case 2: }}$u$ is non-differentiable at $y$. By Proposition \ref{prop:SingularityAndGraph}, $(\Sigma_+[u])^{-1}(y)$ is a non-degenerate closed interval $[x_-,x_+]$. From the proof of Lemma \ref{lemma:SingularityAndNSu}, we can construct an increasing sequence $\{x_n\}_{n\in\mathbb{Z}_+}$ such that $x_n\rightarrow x_-$ and $u$ is differentiable at $y_n:=\Sigma_+[u](x_n)$. By Case 1: \begin{align*}
    x_n=\pi_1\circ F^{-1}(y_n,\d u(y_n)).
  \end{align*} Taking limits, we obtain\begin{align*}
    x_-\in\pi_1\circ F^{-1}(y,\nabla^* u(y)).
  \end{align*}
  Similarly, we derive $x_+\in\pi_1\circ F^{-1}(y,\nabla^* u(y))$. Since $u=T^-[v]$ is semiconcave, by Proposition \ref{prop:PropertyOfSemiconcave} (iv): $$\nabla^*u(y)\subset \partial\nabla^+u(y).$$ Hence: \begin{align*}
    x_-,x_+\in \pi_1\circ F^{-1}(y,\partial\nabla^+u(y)).
  \end{align*}
  Since $\nabla^+u(y)$ is a closed convex subset of $\mathbb{R}$, its boundary contains at most two different points. Thus: \begin{align*}
    \{x_-,x_+\}=\pi_1\circ F^{-1}(y,\partial\nabla^+u(y)).
  \end{align*}
  Observing that  $x_-<x_+$ and using the fact that $p\mapsto\pi_1\circ F^{-1}(y,p)$ is a homeomorphism (by Hypothesis \ref{hypothesis:S2}), we conclude: \[
    [x_-,x_+]=\pi_1\circ F^{-1}(y,\nabla^+u(y)). \qedhere
 \]
\end{proof}

\begin{proposition}
	 If $u=T^-[v]$, $v\in SC(\mathbb{T})$, then the following graphs $$\operatorname{Graph}(\Sigma_+[u]),\ \operatorname{Graph}(\d T^+[u]) \ \text{ and }\  \operatorname{Graph}(\nabla^+u)$$ are topological submanifolds of $\mathbb{R}^2$, related by the commutative diagram:\begin{center}
  \begin{tikzcd}
    & \operatorname{Graph}(\Sigma_+[u]) \arrow[rdd, "\Gamma^+"] \arrow[ldd, "\Gamma^-"'] &                                 \\
    &                                                                                 &                                 \\
  {\operatorname{Graph}(\d T^+[u])} \arrow[rr, "F"] &                                                                                 & \operatorname{Graph}(\nabla^+u)
  \end{tikzcd}
\end{center}
 Here: \begin{itemize}
 	\item $\Gamma^+(x,y)=(y,\partial_2S(x,y))$,
 	\item $\Gamma^-(x,y)=(x,-\partial_1S(x,y))$,
 	\item  $F$ is discrete standard map.
 \end{itemize} Moreover, if $F$ is locally bi-Lipschitz, these three graphs are Lipschitz submanifolds  of $\mathbb{R}^2$.
\end{proposition}

\begin{proof}
{\textbf{Step 1: Verifying Bijectivity of the Diagram.}}

	Both $\Gamma^\pm$ and $F$ are homeomorphisms on $\mathbb{R}^2$, so their restrictions  to the relevant graphs are injective.   We now prove surjectivity. 
	
	\textbf{(i) Surjectivity of $\Gamma^-$:}
	From  Proposition \ref{prop:Regularization}, $T^+[u]$ is differentiable everywhere since $u = T^-[v]$ with $v \in SC(\mathbb{T})$. By Lemma \ref{lemma:RegularityOfSubaction} (b), we have: \begin{equation}\label{eq:-p1S=dT+u}
	-\partial_1S(x,\Sigma_+[u](x))=\d T^+[u](x), \quad \forall x\in\mathbb{R}.
\end{equation}
	Thus: $\Gamma^{-}(x,\Sigma_+[u](x))=(x, \d T^+[u](x))$, implying
	 $\Gamma^-$ is surjective.
	 
	 \textbf{(ii) Surjectivity of $F$:}
	 Let $(y,p)\in\operatorname{Graph}(\nabla^+u)$. From Lemma \ref{lemma:EquivalentDefinitionOfSingularities}, let \begin{equation}\label{eq:pi1Finverse}
	 	x:=\pi_1\circ F^{-1}(y,p).
	 \end{equation} 
	 Then $\Sigma_+[u](x)=y$. Combining the definition of $F$ and the equation \eqref{eq:-p1S=dT+u}, we obtain: \begin{equation}\label{eq:pi2Finverse}
	 	\pi_2\circ F^{-1}(y,p)=-\partial_1S(x,y)=\d T^+[u](x).
	 \end{equation}
	 Hence $$F^{-1}(y,p)=(x,\d T^+[u](x)),$$	  proving surjectivity.
	 
	 {\textbf{Step 2:  Lipschitz Regularity of the Graphs.}}
	 	
	If $F$ is locally bi-Lipschitz, then $\Gamma^\pm$ are locally bi-Lipschitz as well (Remark \ref{remark:F}).  The graph of a Lipschitz  continuous  function $f:\mathbb{R}^n\rightarrow\mathbb{R}$ is a Lipschitz submanifold of $\mathbb{R}^{n+1}$. From Proposition \ref{proposition:SigmaLipschitz}, $\operatorname{Graph}(\Sigma_+[u])$ is a $1$-dimensional Lipschitz submanifold of $\mathbb{R}^2$.
	  	Since the notion of Lipschitz submanifolds is invariant under bi-Lipschitz homeomorphism, we conclude that  $\operatorname{Graph}(\d T^+[u])\cap K$ and $\operatorname{Graph}(\nabla^+u)\cap K$ are Lipschitz submanifolds for any compact set $K\subset \mathbb{R}^2$. By periodicity of $u$ and $T^+[u]$, \(\operatorname{Graph}(\d T^+[u])\) and \(\operatorname{Graph}(\nabla^+ u)\) are Lipschitz submanifolds globally.
\end{proof}

\subsection{Singular dynamics on $\operatorname{Graph}(\nabla^+u)$}

This section analyzes the singular dynamics on \(\operatorname{Graph}( \nabla^+u)\), culminating in the proofs of Theorems \ref{thm:SingularDynamics} and \ref{thm:main3}.

Let $w=T^{-2}[v]$, $u=T^-[v]$, and $v\in SC(\mathbb{T})$.
Define the forward map: \begin{align*}
  \Sigma_+[u]\times \Sigma_+[w]:\operatorname{Graph}(\Sigma_+[u])\rightarrow\operatorname{Graph}(\Sigma_+[w]) 
\end{align*}
given by $ (x,\Sigma_+[u](x))\mapsto\left(\Sigma_+[u](x),\Sigma_+[w](\Sigma_+[u](x))\right).$ This induces the maps \[
  \Omega[u]: \operatorname{Graph}(\d T^+[u])\rightarrow\operatorname{Graph}(\d T^+[w])\]
  and \[\Lambda[u]: \operatorname{Graph}(\nabla^+u)\rightarrow\operatorname{Graph}(\nabla^+w)
\]
illustrated by the commutative diagram:
\begin{center}
  \begin{tikzcd}
    {\operatorname{Graph}(\d T^+[u])} \arrow[dd, "F"'] \arrow[rrrr, "{\Omega [u] }"] &    &  &     & {\operatorname{Graph}(\d T^+[w])} \arrow[dd, "F"] \\
& \operatorname{Graph}(\Sigma_+[u]) \arrow[ld, "\Gamma^+"] \arrow[lu, "\Gamma^-"'] \arrow[rr, "{\Sigma_+[u]\times \Sigma_+[w]}"] & &\operatorname{Graph}(\Sigma_+[w]) \arrow[ru, "\Gamma^-"] \arrow[rd, "\Gamma^+"'] &                                                   \\
\operatorname{Graph}(\nabla^+u) \arrow[rrrr, "{\Lambda [u]}"]  & &  &     & \operatorname{Graph}(\nabla^+w)                  
    \end{tikzcd}
\end{center}
\medskip

\begin{lemma}\label{lemma:Lambda[u]}
	Let $w=T^{-2}[v]$, $u=T^-[v]$, and $v\in SC(\mathbb{T})$. Then the map $\Lambda[u]$
	satisfies the following: \begin{enumerate}
    \item [(i)]  $\Lambda[u](x+1,p)=\Lambda[u](x,p)+(1,0)$.
    \item [(ii)] $\Lambda[u]$ is continuous. If $F$ is locally bi-Lipschitz, then $\Lambda[u]$ is Lipschitz continuous.
    \item [(iii)]  $\Lambda[u]$ does not reverse the orientation of the full pseudo-graph.
    \item [(iv)] The image of a vertical segment under $\Lambda[u]$ is a point that belongs to a vertical segment.
  \end{enumerate}
\end{lemma}

\begin{proof}
For any $(x,p)\in \operatorname{Graph}(\nabla^+u)$, we compute as follows:\begin{align*}
	\Lambda[u](x,p)&=\Gamma^+\circ (\Sigma_+[u]\times \Sigma_+[w])\circ (\Gamma^+)^{-1}(x,p)\\ 
	&=\Gamma^+\circ (\Sigma_+[u]\times \Sigma_+[w])(z,x)\quad \text{ where }(z,x)\text{ satisfies } p=\partial_2S(z,x)\\ 
	&=\Gamma^+(\Sigma_+[u](z),\Sigma_+[w](x))\\ 
	(\text{by Lemma }\ref{lemma:EquivalentDefinitionOfSingularities} )\quad &=\Gamma^+(x,\Sigma_+[w](x))\\ 
	&=(\Sigma_+[w](x),\partial_2S(x,\Sigma_+[w](x))).
\end{align*}
Since  $u=T^-[v]$ is semiconcave,  the properties of $\Lambda[u]$ follow: \begin{itemize}
	\item Item (i) and the first half of the item (ii)  follow directly from Proposition \ref{prop:LiftOfCircleMap}.\\ 
	\item Item (iii) follows from Lemma \ref{lemma:SigmaNondecreasing}.\\ 
	\item  The second half of the item (ii), concerning the Lipschitz continuity of $\Lambda[u]$, follows from Proposition \ref{proposition:SigmaLipschitz}.\\ 
	\item Item (iv) is a consequence of Proposition \ref{prop:SingularityPropagate}, as it describes how singularities propagate under the map $\Sigma_+[u]$.\qedhere
\end{itemize} \end{proof}

\begin{proposition}
	Let $u$ be a discrete weak K.A.M. solution. Then the map $\Lambda[u]:\operatorname{Graph}(\nabla^+u)\righttoleftarrow$	satisfies the following properties: \begin{enumerate}
    \item [(i)] $\Lambda[u](x+1,p)=\Lambda[u](x,p)+(1,0)$.
    \item [(ii)] $\Lambda[u]$ is continuous. If $F$ is locally bi-Lipschitz, then $\Lambda[u]$ is Lipschitz continuous.
    \item [(iii)] $\Lambda[u]$ preserves the orientation of the full pseudo-graph.
    \item [(iv)] The image of a vertical segment under $\Lambda[u]$ is a point that belongs to a vertical segment.
    \item [(v)] $\Lambda [u]$ admits a rotation number for positive iterations: 
    \begin{align*}
      \rho(\Lambda [u]):=\lim_{n\to+\infty}\frac{1}{n}\left( \pi_1\circ\Lambda ^n[u](x,p)-x \right)
    \end{align*}
  exists for all $(x,p)\in\operatorname{Graph}(\nabla^+u)$ and is independent of $(x,p)$. In fact, $\rho(\Lambda[u])$ is the rotation number of the Aubry set.
  \end{enumerate}
\end{proposition}

\begin{proof}
	This proposition follows from Proposition \ref{prop:RotationNumber} and Lemma \ref{lemma:Lambda[u]}.
\end{proof}

Let $u$ be a discrete weak K.A.M. solution. Consider the dynamical system $\Lambda[u]:\operatorname{Graph}(\nabla^+u)\righttoleftarrow$. For any subset $A\subset \operatorname{Graph}(\nabla^+u)$,  the $\alpha$-limit set of $A$ is the set of accumulation points of the negative semiorbit of $\Lambda[u]$ started from $A$: $$\alpha(A):=\bigcap_{n\in\mathbb{Z}_+}\overline{\{\Lambda^i[u](x,p)\mid i\leq -n, (x,p)\in A\}}.$$ 

\begin{proposition}
	Let $u$ be a discrete weak K.A.M. solution. Then the set $\alpha(\operatorname{Graph}(\d u))$
	satisfies the following properties:   \begin{enumerate}
    \item [(i)] The set $\alpha(\operatorname{Graph}(\d u))$ is closed, non-empty and invariant under the dynamics $F$.
    \item [(ii)] the Aubry set $\mathcal{A}^*$ can be expressed as $$\mathcal{A}^*=\bigcap_u\alpha(\operatorname{Graph}(\d u)),$$ where the intersection is taken over all discrete weak K.A.M. solutions.
      \end{enumerate}
\end{proposition}

\begin{proof}
	Since $u$ is a discrete weak K.A.M. solution, $T^-[u]=u+\bar{S}$. 
	
	{\textbf{Step 1.} Inclusion of $\{(x,\d u(x))\mid x\in\mathcal{O}^{n+1}[u]\}\subset\{\Lambda^i[u](x,p)\mid i\leq -n, (x,p)\in \operatorname{Graph}(\d u)\}:$}
	
	For any $x\in\mathcal{O}^{n+1}[u]$, by Remark \ref{remark:coercive} (b), $x$ lies in $(\Sigma_-[u])^n(\mathcal{O}[u])$. Thus, there exists a backward chain $y\in\mathcal{O}[u]$, $y_{-n+1},\dots, y_{-1}$ such that:  $$x\in\Sigma_-[u](y_{-n+1}),\quad y_{-n+1}\in\Sigma_-[u](y_{-n+2}), \quad\cdots \quad y_{-1}\in\Sigma_-[u](y).$$
	Using Lemma \ref{lemma:Sigma-andSigma+}, this chain can be reversed under $\Sigma_+[u]$: $$y_{-n+1}\in\Sigma_+[u](x),\quad y_{-n+2}\in\Sigma_+[u](y_{-n+1}), \quad\cdots \quad y\in\Sigma_+[u](y_{-1}),$$
	 leading to: $$(x,\d u (x))\in\Lambda^{-n}[u](y,\d u(y)).$$
	 
	{\textbf{Step 2.} Inclusion of $\{\Lambda^i[u](x,p)\mid i\leq -n, (x,p)\in \operatorname{Graph}(\d u)\} \subset \{(x,\d u(x))\mid x\in\mathcal{O}^{n}[u]\}:$}
	
	By Proposition \ref{prop:SingularityPropagate}, $u$ remains differentiable along the backward iterates of $\lambda[u]$.
	 If $u$ is differentiable at $y$,  $u$ is differentiable on $$\bigcup_{i\leq 0}\Lambda^i[u](y,\d u(y)).$$ So if $$(x,\d u(x))\in\bigcup_{i\leq-n}\Lambda^{i}[u](y,\d u(y)),$$ the corresponding chain of points $y_{-n+1},\dots, y_{-1}$ satisfy $$y_{-n+1}\in\Sigma_+[u](x),\quad y_{-n+2}\in\Sigma_+[u](y_{-n+1}), \quad\cdots \quad y\in\Sigma_+[u](y_{-1}),$$
	and $u$ is differentiable at $y_{-n+1},\dots, y_{-1}$. By Lemma \ref{lemma:SingularityAndNSu} and Proposition \ref{prop:SingularityAndGraph}, we have $$x\in\Sigma_-[u](y_{-n+1}),\quad y_{-n+1}\in\Sigma_-[u](y_{-n+2}), \quad\cdots \quad y_{-1}\in\Sigma_-[u](y).$$ This implies $x\in\mathcal{O}^n[u]$.

{\textbf{Step 3.} Equivalence of \(\alpha(\operatorname{Graph}(\d u))\) and \(\{(x, \d u(x)) \mid x \in \mathcal{O}^\infty[u]\}\):}

	Using the inclusions from Steps 1 and 2, we deduce: \begin{align*}
		\bigcap_{n\in\mathbb{Z}_+}\{(x,\d u(x))\mid x\in\mathcal{O}^{n}[u]\}=\bigcap_{n\in\mathbb{Z}_+}\overline{\{\Lambda^i[u](x,p)\mid i\leq -n, (x,p)\in \operatorname{Graph}(\d u)\}}.	\end{align*}	
	Thus, $$\alpha(\operatorname{Graph}(\d u))=\{(x,\d u(x))\mid x\in\mathcal{O}^{\infty}[u]\},$$
	showing by  Proposition \ref{prop:invariant} it is  closed, non-empty and invariant. By Proposition \ref{prop:defofAubry},  \[\bigcap_{u \text{ weak K.A.M.}}\alpha(\operatorname{Graph}(\d u))=\bigcap_{u \text{ weak K.A.M.}}\left\{(x,p) \mid x\in\mathcal{O}^{\infty}[u],p=\d u(x)\right\}=\mathcal{A}^*.\]
	
This completes the proof of both parts of the proposition.
	\end{proof}

\begin{appendix}

\section{Semiconcave functions}\label{sec:appendixA}

Let us recall some fundamental definitions and properties  on semiconcave functions (see \cite{Cannarsabook, Ber07} for instance).

\begin{definition} Let $u:\mathbb{R}^d\rightarrow\mathbb{R}$ be a continuous function.
    \begin{itemize}
        \item For any $x\in \mathbb{R}^d$, the sets $$
        \begin{aligned}
        & \nabla^{-} u(x)=\left\{p \in (\mathbb{R}^d)^*: \liminf _{y \rightarrow x} \frac{u(y)-u(x)-p(y-x)}{\|y-x\|} \geq 0\right\} \\
        & \nabla^{+} u(x)=\left\{p \in (\mathbb{R}^d)^*: \limsup _{y \rightarrow x} \frac{u(y)-u(x)-p(y-x)}{\|y-x\|} \leq 0\right\}
        \end{aligned}
        $$
        are called, respectively, the \emph{superdifferential} and \emph{subdifferential} of $u$ at $x$. From the definition it follows that, for any $x\in\mathbb{R}^d$, $-\nabla^-u(x)=\nabla^+(-u)(x).$
        \item Let $\Omega\subset \mathbb{R}^d$ be a  convex open set. We say that $u$ is \emph{semiconcave} (with a linear modulus) on $\Omega$ if there exists a constant $C\geq 0$ such that \begin{align*}
            \lambda u(x)+(1-\lambda) u(y)-u(\lambda x+(1-\lambda) y) \leqslant \frac{C}{2} \lambda(1-\lambda)\|x-y\|^2
                \end{align*}
                for any $ x, y \in\Omega$ and $\lambda\in[0,1]$. Then $C$ is called the \emph{semiconcavity constant} of $u$. A function $v$ is called \emph{semiconvex} on $\Omega$ if $-v$ is semiconcave. We say that $u$ is \emph{locally  semiconcave} (with a linear modulus) if for each $x\in\mathbb{R}^d$, there exists a convex open neighborhood $U_x$ such that $u$ is semiconcave on $U_x$.
        \item For any $x\in\mathbb{R}^d$, a covector $p\in(\mathbb{R}^d)^*$ is called a reachable differential of $u$ at $x$ if there exists a sequence $\{x_n\}\subset \mathbb{R}^d\backslash\{x\}$ such that  $u$ is differentiable at $x_k$ for each $k\in\mathbb{Z}_+$, and \begin{align*}
          \lim _{k \rightarrow +\infty} x_k=x, \quad \lim _{k \rightarrow +\infty} \d u\left(x_k\right)=p .
        \end{align*}
        The set of all reachable gradients of $u$ at $x$ is denoted by $\nabla^*u(x)$.
    \end{itemize} 
\end{definition}

One can then have the following well-known facts.
\begin{proposition}\label{prop:PropertyOfSemiconcave}Let $u:\mathbb{R}^d\rightarrow\mathbb{R}$ be a continuous function, $\Omega\subset \mathbb{R}^d$ be a  convex open set and $x, x_0\in\Omega$.
    \begin{enumerate}
        \item [(i)] 
        Let $u:\mathbb{R}^d\rightarrow\mathbb{R}$ be a semiconcave function on $\Omega$ with a linear modulus and some semiconcave constant $C$. Then \begin{align*}
            u(x)-u(x_0)\leq p_0(x-x_0)+\frac{C}{2}\|x-x_0\|^2
        \end{align*}
        for any $p_0\in \nabla^+u(x_0)$. Conversely, 
        if there exists  $C\geq  0$ and $p_0\in(\mathbb{R}^d)^*$ such that \begin{align*}
            u(x)-u(x_0)\leq p_0(x-x_0)+\frac{C}{2}\|x-x_0\|^2, \ \quad \forall x\in\Omega
        \end{align*}
        then $u$ is semiconcave with semiconcavity constant $C$ and  $p_0\in \nabla^+u(x_0)$.
        \item[(ii)] 
        Let $v:\mathbb{R}^d\rightarrow\mathbb{R}$ be another continuous function. If $u\leq v$ on $\Omega$ and $u(x)=v(x)$, then $\nabla^+u(x)\supset   \nabla^+v(x)$ and $\nabla^-u(x)\subset   \nabla^-v(x)$.
        \item[(iii)] 
         $u$ is differentiable at $x$ if and only if $\nabla^+u(x)$ and $\nabla^-u(x)$ are both nonempty; in
        this case, we have that
        $$\nabla^+u(x)=\nabla^-u(x)=\{\d u(x)\}.$$
        \item[(iv)] 
        Let $u:\mathbb{R}^d\rightarrow\mathbb{R}$ be a semiconcave function on $\Omega$. Then  $\nabla^*u(x)\subset \partial \nabla^+u(x)$ for all $x\in\Omega$, where $\partial$ is the boundary taken with respect to the standard topology of $\mathbb{R}^d$.
        \item[(v)] 
        Let $u:\mathbb{R}^d\rightarrow\mathbb{R}$ be a semiconcave function on a compact set $\Omega$. Then the gradient of $u$, defined almost everywhere in $\Omega$, has bounded variation. Hence the gradient of $u$ is bounded in $\Omega$.
    \end{enumerate}
\end{proposition}

\section{Twist maps of the $2$-dimensional annulus}\label{sec:appendixB}

We begin by introducing the notion of twist maps. 

\begin{definition}\label{definition:TwistMap}
A \emph{twist map} of the $2$-dimensional annulus is a $C^1$-diffeomorphism $f:\mathbb{T}\times \mathbb{R}^*\righttoleftarrow$ such that \begin{enumerate}
  \item $f$ is isotopic to identity ;
  \item $f$ is exact symplectic, i.e. if $f(\theta,p)=(\theta',p')$, then the $1$-form $p'd\theta'-pd\theta$ is exact;
  \item $f$ twists verticals to the right: if $F:\mathbb{R}\times \mathbb{R}^*\righttoleftarrow$, denoted by $F(x, p)=(y, p')$, is a lift of $f$, then there exists $\varepsilon>0$ such that\begin{align*}
    0<\frac{\partial y}{\partial p}(x,p)< \frac{1}{\varepsilon}, \quad \forall (x,p)\in \mathbb{R}\times \mathbb{R}^*.
  \end{align*}
\end{enumerate}
\end{definition}
\begin{definition}\label{definition:GeneratingFunction}
  A \emph{generating function} of the lift $F:\mathbb{R}\times \mathbb{R}^*\righttoleftarrow$ is a $C^2$ function $S:\mathbb{R}^2\rightarrow\mathbb{R}$ such that \begin{enumerate}
    \item $S$ is diagonal-periodic:\begin{align*}
      \forall (x,y)\in\mathbb{R}^2, \quad S(x+1,y+1)=S(x,y);
    \end{align*}
    \item  there exists $\varepsilon>0$ such that\begin{align*}
      \frac{\partial^2 S}{\partial x\partial y}(x,y)<-\varepsilon, \quad \forall (x,y)\in \mathbb{R}\times \mathbb{R}.
    \end{align*}
    \item for all $(x,p), (y,p')\in\mathbb{R}^2$,  we have the following equivalence\begin{align*}
      F(x,p)=(y,p')\Longleftrightarrow\begin{cases}
        p=-\partial_1S(x,y)\\ 
        p'=\partial_2S(x,y)
      \end{cases}.
    \end{align*}
  \end{enumerate}
\end{definition}

\begin{remark}
  For any lift $F$ of a twist map $f$, there exists a unique generating function $S$, up to addition by a constant, that satisfies Definition \ref{definition:GeneratingFunction} (see \cite[Appendix 1]{MacKey89}).
\end{remark}

\begin{proposition}\label{prop:ScHypothesis}
  Let $f$ be a twist map and $S$ be a generating function. For any $c\in \mathbb{R}$, the function $$S_c(x,y) := S(x, y) + c(x-y) $$ satisfies  Hypothesis \ref{hypothesis:S1}, Hypothesis \ref{hypothesis:S2} and Hypothesis \ref{hypothesis:S3}. 
\end{proposition}

\begin{proof}
The diagonal-periodicity and ferromagnetic property of  $S_c$  follows directly from items (1) and (3) of Definition \ref{definition:GeneratingFunction}, respectively. Since $S$ is of class $C^2$, it follows that $S_c$ is of class $C^1$ and locally semiconcave (see \cite[Proposition 2.1.2]{Cannarsabook}). The remaining task is to demonstrate that $S_c$ is coercive.  This property follows from the superlinearity of the generating function $S$, expressed as \begin{align*}
    \lim_{|x-y|\rightarrow+\infty} \frac{S(x,y)}{|x-y|}=+\infty.
  \end{align*}
  This conclusion is an immediate consequence of a lemma from the work of  MacKay, Meiss and Stark \cite[Lemma A1.6]{MacKey89}. Furthermore, the Non-crossing Lemma \cite{aubry1983discrete} can be obtained by integrating $\frac{\partial^2 S}{\partial x\partial y}$ over a quadrangle; we omit its detailed proof here. 
 \end{proof}

 \begin{remark}
  \begin{enumerate}
    \item [(a)] The function $S_c$ is referred to as the generating function at the cohomology class $c$ of $\mathbb{T}$. As  proved  in Proposition \ref{prop:ScHypothesis}, all theories in this paper are applicable to   $S_c$. For instance, consider any $c$ in the first cohomology group $H^1(\mathbb{T},\mathbb{R})$, which is isomorphic to $\mathbb{R}$, the Discrete Weak K.A.M. Theorem \ref{theorem:DWKT} defines the effective interaction $\overline{S_c}$. The function $\alpha(c)=-\overline{S_c}$ is called  Mather's $\alpha$-function, which is  $C^1$. As a second example, the Aubry set $\mathcal{A}_c^*$ at cohomology class $c$ is $F_c$-invariant, where \begin{align*}
      F_c(x,p)=F(x,p+c)-(0,c).
    \end{align*}
    This corresponds to the discrete standard map associated with $S_c$.
    Consequently,  the set $(0,c)+\mathcal{A}_c^*$ is $F$-invariant, where $F$ is the lift of the twist map. Thus, by examining  the discrete weak K.A.M. theory of $S_c$, we can  investigate the dynamics of $F$ at the cohomology class  $c$.
    \item [(b)] Since the dimension $d=1$, the the Aubry set $\mathcal{A}_c^*$ at cohomology class $c$ possesses additional properties beyond Proposition \ref{prop:PropertyOfAubry} (see \cite[Theorem 11.10]{GT11}):  \begin{itemize}
      \item $(0,c)+\mathcal{A}^*_c$ is an $F$-ordered set, i.e., $\forall z,z'\in (0,c)+\mathcal{A}^*_c$, \begin{align*}
        \pi_1(z)<\pi_1(z')\quad \Longrightarrow \quad \pi_1(F(z))<\pi_1(F(z')).
      \end{align*}
      \item The rotation number $\rho(c)$ of $(0,c)+\mathcal{A}_c^*$ exists, i.e.,  the limit $\Lim{n\to\pm\infty}\frac{\pi_1(F^n(x,p))-x}{n}$ exists and is independent of $(x,p)\in (0,c)+\mathcal{A}^*_c$. In fact, $\rho(c)=\alpha'(c)$.
    \end{itemize}
  \end{enumerate}
 \end{remark}

\section{Proofs of Lemmas~\ref{lemma:RegularityOfSubaction} and \ref{lemma:WellDefined}}\label{sec:appendixC}
For selfcontainedness, we give proofs for Lemmas~\ref{lemma:RegularityOfSubaction} and \ref{lemma:WellDefined}.

\begin{proof}[Proof of Lemma~\ref{lemma:RegularityOfSubaction}]

(a). Suppose that $T^-[u](y)=u(x)+S(x,y)$ for some $(x,y)\in\mathbb{R}^{d}\times\mathbb{R}^{d}$. 
      
      To establish the first inclusion, $-\partial_1S(x,y)\in \nabla^-u(x)$, we begin with the definition of the discrete backward Lax-Oleinik operator : $T^-[u](y)=\inf_{z\in\mathbb{R}^d}u(z)+S(z,y)$. Then we have   \begin{align*}
          u(x)+S(x,y)\leq u(z)+S(z,y),\ \quad \forall z\in\mathbb{R}^d.
      \end{align*}
      Let $\Omega$ be a bounded convex open neighborhood of $x$. Since $S$ is locally semiconcave, we have $S(\cdot,y)$ is semiconcave on $\Omega$. 
      By utilizing item (i)   of Proposition~\ref{prop:PropertyOfSemiconcave}, there exists $C\geq 0$ such that  \begin{align*}
          S(z,y)-S(x,y)\leq \partial_1S(x,y)(z-x)+\frac{C}{2}\|z-x\|^2
      \end{align*}
      for any $z\in\Omega$. Combining these two inequalities,  we get \begin{align*}
          u(x)-u(z)\leq \partial_1S(x,y)(z-x)+\frac{C}{2}\|z-x\|^2,\ \quad \forall z\in\Omega.
      \end{align*}
      Hence, $\partial_1S(x,y)\in \nabla^+(-u)(x)$ follows from item (i) of Proposition \ref{prop:PropertyOfSemiconcave}. Since $\nabla^+(-u)(x)=-\nabla^-u(x)$, we conclude that $-\partial_1S(x,y)\in\nabla^-u(x)$.

      Now, we move on to the second inclusion,  $\partial_2S(x,y)\in\nabla^+T^-[u](y)$. 
      According to the definition of the discrete backward Lax-Oleinik operator $T^-$, we have  $T^-[u](\cdot)\leq u(x)+S(x,\cdot)$. Hence it follows from  items (ii) and (iii) of Proposition \ref{prop:PropertyOfSemiconcave} that \begin{align*}
          \{\partial_2S(x,y)\}=\nabla^+ (u(x)+S(x,\cdot))(y)\subset \nabla^+T^-[u](y).
      \end{align*}
      
(b). The proof for item (b) follows from a similar argument as that of (a).
\end{proof}

\begin{proof}[Proof of Lemma~\ref{lemma:WellDefined}]
    (a). Suppose there exists $x_1\ne x_2$ such that  $T^-[u](y)=u(x_1)+S(x_1,y)$ and $T^-[u](y)=u(x_2)+S(x_2,y)$.
          Using (a) of Lemma \ref{lemma:RegularityOfSubaction}, we get: \begin{align*}
               \{\partial_2S(x_1,y), \partial_2S(x_2,y)\}\subset \nabla^+T^-[u](y).
           \end{align*}
           The ferromagnetic condition of $S$ asserts that if $x_1\ne x_2$, then $\partial_2S(x_1,y)\ne \partial_2S(x_2,y)$. Hence $\nabla^+T^-[u](y)$ is not a singleton set. It follows from item (iii) of Proposition \ref{prop:PropertyOfSemiconcave} that $T^-[u]$ is not differentiable at $y$. 
      
           Thanks to item (a) of Remark \ref{remark:coercive}, we only need to demonstrate that  if $x$ is the unique point such that  $T^-[u](y)=u(x)+S(x,y)$, then $T^-[u]$ is differentiable at $y$. 
          
           For any converging sequence $y_k\rightarrow y$, let us pick $x_k$ such that $T^-[u](y_k)=u(x_k)+S(x_k,y_k)$. Since $x$ is the unique point such that  $T^-[u](y)=u(x)+S(x,y)$, we have $x_k\rightarrow x$ due to the continuity of functions $T^-[u]$, $u$ and $S$. Then \begin{align*}
               T^-[u](y_k)&=u(x_k)+S(x_k,y)+\left(S(x_k,y_k)-S(x
               _k,y)\right)\\ 
               &\geq T^-[u](y)+\left(S(x_k,y_k)-S(x
               _k,y)\right)\\ 
               &=T^-[u](y)+\partial_2S(x_k,y)(y_k-y)+o(\|y_k-y\|)
           \end{align*}
           Hence $\partial_2S(x,y)\in\nabla^-T^-[u](y)$ by the definition of superdifferential. Using item  (a) of Lemma~\ref{lemma:RegularityOfSubaction} and item (iii) of Proposition \ref{prop:PropertyOfSemiconcave}, we have that  $T^-[u]$ is differentiable at $y$.
           
           (b). The proof for item (b) follows from a similar argument as that of (a).
           \end{proof}

\end{appendix}

\section*{Acknowledgments}
The authors thank Prof. P. Thieullen for a careful reading of the manuscripts and many helpful suggestions during his visit at Beijing Normal University.


\bibliographystyle{alpha}
\bibliography{reference}

\end{document}